\documentclass{amsart}
\usepackage{MnSymbol}
\usepackage{bbold}
\usepackage[mathscr]{euscript}

\usepackage{hyperref}
\hypersetup{
    colorlinks,
    citecolor=black,
    filecolor=black,
    linkcolor=black,
    urlcolor=black
}
\usepackage{comment}

\DeclareFontFamily{OT1}{pzc}{}
\DeclareFontShape{OT1}{pzc}{m}{it}{<-> s * [1.100] pzcmi7t}{}
\DeclareMathAlphabet{\mathpzc}{OT1}{pzc}{m}{it}

    \usepackage{etoolbox}
    \patchcmd{\section}{\scshape}{\large\bfseries}{}{}
    \makeatletter
    \renewcommand{\@secnumfont}{\bfseries}
    \makeatother

\usepackage{tikz-cd}
\tikzcdset{arrow style=math font}
\tikzcdset{scale cd/.style={every label/.append style={scale=#1},
    cells={nodes={scale=#1}}}}

\usepackage[maxbibnames=99]{biblatex}
\usepackage{booktabs}
\bibliography{references}

\numberwithin{equation}{section}
\newtheorem{theorem}{Theorem}[section]
\newtheorem*{theorem*}{Theorem}
\newtheorem{corollary}[theorem]{Corollary}
\newtheorem{lemma}[theorem]{Lemma}
\newtheorem{proposition}[theorem]{Proposition}
\theoremstyle{definition}

\newtheorem*{question*}{Question}

\newtheorem{remark}[theorem]{Remark}
\newtheorem{example}[theorem]{Example}

\def\ZZ{\mathbb{Z}}
\def\RR{\mathbb{R}}
\def\AA{\mathcal{A}}
\def\VV{\mathcal{V}}
\def\CC{\mathcal{C}}
\def\DD{\mathcal{D}}
\def\FF{\mathcal{F}}
\def\GG{\mathcal{G}}

\def\mono{\rightarrowtail}
\def\epi{\twoheadrightarrow}

\def\Hom{\mathrm{Hom}}
\def\Im{\mathrm{Im}}

\def\Coker{\mathrm{Coker}}
\def\1{\mathbb{1}}
\def\pt{\mathsf{pt}}
\def\Triv{\mathrm{Triv}}
\def\Coinv{\mathrm{Coinv}}
\def\Inv{\mathrm{Inv}}
\def\EE{\mathcal{E}}
\def\KK{\mathbb{K}}

\def\MC{\mathrm{MC}}
\def\MH{\mathrm{MH}}

\def\LL{\mathcal{L}}
\def\Mod{\mathrm{Mod}}
\def\WW{\mathcal{W}}
\def\Fr{\mathrm{Fr}}
\def\CoFr{\mathrm{CoFr}}
\def\Forg{\mathrm{Forg}}
\def\NN{\mathscr{N}}
\def\DMod{\mathrm{DistMod}}
\def\GrMod{\mathrm{GrMod}}
\def\GGG{\mathbb{G}}
\def\Rep{\mathrm{Rep}}
\def\GrRep{\mathrm{GrRep}}
\def\Ext{\mathrm{Ext}}
\def\Tor{\mathrm{Tor}}

\let\oldtocsection=\tocsection 
\let\oldtocsubsection=\tocsubsection 
\renewcommand{\tocsection}[2]{\hspace{0mm}\oldtocsection{#1}{#2}}
\renewcommand{\tocsubsection}[2]{\hspace{1em}\oldtocsubsection{#1}{#2}}

\title{Magnitude homology is a derived functor}

\author{Yasuhiko Asao}\address{
Faculty of Applied Mathematics, Fukuoka University}
\email{asao@fukuoka-u.ac.jp}
\author{Sergei O. Ivanov} \address{
Beijing Institute of Mathematical Sciences and Applications}
\email{ivanov.s.o.1986@gmail.com, ivanov.s.o.1986@bimsa.cn}

\thanks{The first named author thanks S. Wakatsuki for a fruitful discussion.}

\begin{document}

\begin{abstract}
We prove that the magnitude (co)homology of an enriched category can, under some  technical assumptions, be described in terms of derived functors between certain abelian categories. We show how this statement is specified for the cases of quasimetric spaces, finite quasimetric spaces, and finite digraphs. For quasimetric spaces, we define the notion of a distance module over a quasimetric space, define the functor of (co)invariants of a distance module and show that the magnitude (co)homology can be presented via its derived functors. As a corollary we obtain that the magnitude cohomology of a quasimetric space can be presented in terms of Ext functors in the category of distance modules. For finite quasimetric spaces, we show that magnitude (co)homology can be presented in terms of Tor and Ext functors over a certain graded algebra, called the distance algebra of the quasimetric space. For finite digraphs, the distance  algebra is a bound quiver algebra. In addition, we show that the magnitude cohomology algebra of a finite quasimetric space can be described as a Yoneda algebra.
\end{abstract}

\maketitle

\tableofcontents

\section{Introduction}

In \cite{leinster2008euler} Leinster introduced the notion of the Euler characteristic of a finite category. This Euler characteristic is a rational number coinciding with the Euler characteristic of the  classifying space if it is a finite CW-complex. Later \cite{leinster2013magnitude} he generalized this notion to the notion of the magnitude of a finite $\VV$-category $\CC,$ where $\VV$ is a monoidal category equipped with a multiplicative  homomorphism to a ring 
$|\cdot |:{\rm Ob}(\VV) \to \KK$. A special case of this notion is the notion of the magnitude of a finite metric space. The concept of magnitude was further explored in the special case of finite graphs \cite{leinster2019magnitude}, where the magnitude of a graph is defined as a formal power series with integer coefficients. It was shown that the coefficients of the formal power series can be described as a certain alternating sum, similar to the sum arising in the classical definition of the Euler characteristic. Having this analogy in mind, Hepworth and Willerton \cite{hepworth2017categorifying} defined the magnitude homology of a graph and showed that the coefficients of the magnitude power series are alternating sums of ranks of the magnitude homology groups. Leinster and Shulman \cite{leinster2021magnitude} generalized the notion of the magnitude homology from the case of a graph to the case of a small $\VV$-category $\CC,$ where $\VV$ is a semicartesian symmetric monoidal category equipped with a strong monoidal functor to an abelian symmetric monoidal category $\Sigma:\VV\to \AA.$ As in the case of magnitude, a special case was that of metric space. They showed that the relation between the magnitude and the magnitude homology can be generalized from the case of graphs to the case of finite metric spaces. The dual theory of magnitude cohomology was developed by Hepworth  \cite{hepworth2022magnitude}, where he introduces a product in the magnitude cohomology and shows that in the case of a finite metric space the magnitude cohomology algebra completely determines the metric space. 

If $\VV={\sf Set}$ is the category of sets, $\AA={\sf Ab}$ is the category of abelian groups and $\Sigma$ is the functor of free abelian group, then a $\VV$-category is just a category, and the magnitude homology of a category coincides with its ordinary homology  (homology of the classifying space). It is well known that the ordinary homology of a category can be represented as derived functors of the functor of a colimit \cite[App.II,Prop.3.3]{gabriel2012calculus}, \cite[\S1]{quillen2006higher}. To be more precise, if we denote by $\Mod(\CC)$ the category of functors to the category of abelian groups $\CC\to {\sf Ab},$ then we obtain a functor of colimit 
${\rm colim}:\Mod(\CC)\to  {\sf Ab},$
which is a right exact functor between abelian categories, and the homology of $\CC$ can be presented via it's left  derived functors  
\[H_n(\CC)\cong (L_n{\rm colim})(\ZZ).\] 
There is a similar statement about the cohomology of $\CC$ and the functor of limit. The goal of this work is to generalize these statements to the magnitude (co)homology of an enriched category, and point out the specifics that arise in the special cases of quasimetric space, finite quasimetric space and digraph. We also study the magnitude cohomology algebra of a finite quasimetric space and show that it can be presented as a Yoneda algebra. 

In order to prove our main statement about the magnitude homology of an enriched category, we assume that $\AA$ is an abelian B\'enabou cosmos and develop the theory of $\CC$-modules where $\CC$ is a small $\AA$-category. The category of $\CC$-modules $\Mod(\CC)$ is defined as the category of $\AA$-functors $\CC \to \AA$. We show that $\Mod(\CC)$ is an abelian category with enough of projectives and injectives, assuming that $\AA$ has enough of projectives and injectives. By an augmented $\AA$-category we mean a small $\AA$-category $\CC$ equipped with a functor into the  single-object $\AA$-category $\CC \to \pt.$ Similar to how this is done in the group homology theory, for an augmented $\AA$-category $\CC$ we define functor of trivial module, and its left and right adjoint functors of coinvariants and invariants  
\[ \Coinv,\Inv : \Mod(\CC) \longrightarrow \AA: \Triv, \hspace{1cm} \Coinv \dashv \Triv \dashv \Inv \]
and define the homology and cohomology of $\CC$ via their derived functors 
\[ H_n(\CC) = (L_n\Coinv)(\Triv(\1)), \hspace{1cm} H^n(\CC) = (R^n\Inv)(\Triv(\1)).\]
We prove that, under some additional assumptions on $\CC,$ this homology and cohomology can be computed using some version of the bar resolution. 

Returning to the setting of the magnitude homology of enriched categories, we assume that $\VV$ is a B\'enabou cosmos, $\AA$ is an abelian B\'enabou cosmos and $\Sigma:\VV\to \AA$ is a strong monoidal functor. Then for an augmented $\VV$-category $\CC$, satisfying some technical conditions, we prove that the magnitude (co)homology of $\CC$ can be presented as (co)homology of the corresponding $\AA$-category $\Sigma \CC,$ obtained from $\CC$ using the cosmos change construction 
\[ \MH^\Sigma_n(\CC) \cong  H_n(\Sigma \CC), \hspace{1cm} \MH_\Sigma^n(\CC) \cong  H^n(\Sigma\CC). \]
Let us note that for our convenience here we have slightly generalized the definition of the magnitude homology, and abandoned the assumption that the monoidal category $\VV$ is semicartesian, replacing this assumption with the assumption that an augmentation is given on $\CC$.

We also studied how these concepts are specialized for the case of quasimetric spaces. For a quasimetric space $X$ and a commutative ring $\KK$ we define the notion of a distance module over $X.$ Roughly speaking, a distance module over $X$ consists of a collection of $\RR$-graded $\KK$-modules indexed by points of $X$ equipped with some $\KK$-linear maps between them, satisfying some axioms. 
We prove that the category of distance modules over $X$ is equivalent to the category $\Sigma X$-modules
\[ \DMod(X)\simeq \Mod(\Sigma X).\]
Using this, we give a more explicit description of the magnitude homology of a quasimetric space in terms of derived functors. As a corollary we obtain a description of the magnitude cohomology with coefficients in  $\KK$ in terms of Ext-functors in the abelian category of distance modules
\[ \MH^{n,\ell}(X) = \Ext^n(\tilde \KK, \tilde  \KK[\ell]),\]
where $\tilde \KK$ is the trivial distance module and $\tilde \KK[\ell]$ is this distance module with a shifted grading. 

If $X$ is a finite quasimetric space, using these descriptions we show that the magnitude (co)homology can be described in terms of bigraded Tor and Ext functors over some $\RR$-graded algebra $\sigma X,$ called the distance algebra of $X$
\[\MH_{n,\ell}(X)\cong  \Tor_{n,\ell}^{\sigma X}(S,S), \hspace{1cm} \MH^{n,\ell}(X)\cong \Ext^{n,\ell}_{\sigma X}(S,S).\]
Here $S$ denotes a $\sigma X$-module described as a quotient of $\sigma X$ by some explicit ideal. Moreover, in this case the multiplication on the magnitude cohomology can be described in terms of the Yoneda product. 

We also have some versions of all the statements about distance modules which are relative to some subgroup $\GGG\subseteq \RR.$ Namely, if $\GGG\subseteq \RR$ is a subgroup of the additive group of real numbers and $X$ is a quasimetric space such that all finite distances lie in $\GGG,$ then we have a notion of $\GGG$-graded distance module, $\GGG$-graded algebra $\sigma X$ and so on. This is especially useful when $\GGG=\ZZ$ and all finite distances are integers. 

Specifying to an even more special case of a finite digraph $G$ treated as a quasimetric space, we find that the distance algebra can be presented as the quotient algebra of the path algebra 
\[\sigma G\cong \KK G/R(G)\] 
by an admissible ideal $R(G)$ defined by some explicit set of relations. In this case the $\sigma G$-module $S$ can be presented as $S \cong \KK G/J,$ where $J$ is the ideal generated by all arrows. Moreover in this case the category of $\ZZ$-graded distance modules can be described as the category of  graded representations of the bound quiver $(G,R(G))$ 
\[
\DMod_\ZZ(G) \simeq \GrRep(G,R(G)).
\]

\section{Background from   enriched category theory} 

In this section we collect some definitions and results from enriched category theory that can be found in \cite{kelly1982basic}, \cite{eilenberg1966closed} and \cite{lawvere1973metric}. 

\subsection{B\'enabou cosmos}

Further in this section, unless otherwise stated, we denote by $\VV$ a B\'enabou cosmos i.e. a complete, cocomplete symmetric closed monoidal category. The underlying category is denoted by $\VV_0.$ The tensor product in $\VV$ is denoted by $\otimes$ and the unit object is denoted by $\1.$ The inner hom-object is denoted by $\VV(-,=),$ while the hom-set is denoted by $\VV_0(-,=).$

Since the functors $-\otimes v$ and $\VV(v,-)$ are adjoint, we have the unit and counit of the adjunction 
\begin{equation}\label{eq:counits}
u\longrightarrow \VV(v,u\otimes v),  \hspace{1cm}    \VV(v,u)\otimes v\longrightarrow u.
\end{equation}
This counit allows to define the composition law for the  inner hom of $\VV$
\begin{equation}\label{eq:composition}
\circ :   \VV(u,w)\otimes  \VV(v,u)\longrightarrow \VV(v,w) 
\end{equation}
as the adjoint to the composition of counits (see \cite[\S 1.6]{kelly1982basic})
\begin{equation}
\VV(u,w)\otimes(\VV(v,u)\otimes v )    \longrightarrow  \VV(u,w)\otimes u \longrightarrow w.
\end{equation}

If $\CC$ is a $\VV$-category, the hom-object is denoted by $\CC(c,c'),$ and the hom-set in $\CC$ is denoted by $\CC_0(c,c')=\VV_0(\1,\CC(c,c')).$  The underlying category of $\CC$ will be also denoted by $\CC_0.$ 

The category $\VV$ itself can be regarded as a $\VV$-category with the inner hom $\VV(u,v),$ the composition defined by  \eqref{eq:composition} and unit morphism $1_v:\1\to \VV(v,v)$ is defined by the unit of the adjunction  \eqref{eq:counits} for $v=\1$
(see \cite[\S 1.6]{kelly1982basic}).

\subsection{Ends}

For $\VV$-categories $\CC$ and $\DD$ we can consider their tensor product $\CC\otimes \DD$ and the opposite category $\CC^{op}$ \cite[\S 1.4]{kelly1982basic}. 

Let $\FF:\CC^{op}\otimes \CC\to \DD$ be a $\VV$-functor. We say that a collection of morphisms 
$\varphi_c \in \DD_0(d, \FF(c,c))$ indexed by objects of $\CC$   is $\VV$-dinatural (or extraordinary $\VV$-natural), if the diagram 
\begin{equation}\label{eq:dinatural}
\begin{tikzcd}
\CC(c,c')\ar[r,"{\FF(c,-)}"] \ar[d,"{\FF(-,c')}"] & \DD(\FF(c,c),\FF(c,c')) \ar[d,"{\DD(\varphi_c,1)}"] \\ 
\DD(\FF(c',c'),\FF(c,c')) \ar[r,"{\DD(\varphi_{c'}, 1)}"] & \DD(d,\FF(c,c'))
\end{tikzcd}
\end{equation}
is commutative (see \cite[\S 1.7]{kelly1982basic}). The universal $\VV$-dinatural transformation is called the end of $\FF$, and its domain $d$ is denoted by $\int_\CC \FF=\int_c \FF(c,c).$

If $\CC$ is small and $\DD=\VV,$ then the end exists and can be computed as an equaliser  \cite[\S 2.1]{kelly1982basic}
\begin{equation}\label{eq:end}
\int_\CC \FF \to \prod_{c} \FF(c,c)  \rightrightarrows \prod_{c,c'} \VV(\CC(c,c'),\FF(c,c')),
\end{equation}
where the morphisms $\theta,\theta': \prod_{c} \FF(c,c)  \rightrightarrows \prod_{c,c'} \VV(\CC(c,c'),\FF(c,c'))$ are defined by the equations ${\sf pr}_{c,c'}\theta = \theta_{c,c'} {\sf pr}_c$ and ${\sf pr}_{c,c'}\theta = \theta_{c,c'} {\sf pr}_{c'},$ and the morphisms $\theta_{c,c'}$ and $\theta'_{c,c'}$ are double adjoint morphisms to $\FF(c,-)$ and $\FF(-,c')$ from the diagram \eqref{eq:dinatural}.

The coend $\int^\CC \FF=\int^c\FF(c,c)$ is defined dually, and if $\CC$ is small and $\DD=\VV,$ it exists and can be presented as a coequaliser 
\begin{equation}\label{eq:coend}
\coprod_{c,c'} \CC(c,c')\otimes \FF(c',c) \rightrightarrows \coprod_c \FF(c,c) \to \int^\CC \FF.
\end{equation} 

\subsection{Functor categories}

Assume that $\CC$ and $\DD$ are $\VV$-categories. For $\VV$-functors $\FF,\GG:\CC\to \DD$ a $\VV$-natural transformation $\varphi:\FF\to \GG$ is defined as a collection of morphisms $\varphi_c\in \DD_0(\FF c,\GG c)$ indexed by objects of $\CC$ such that the diagram 
\begin{equation}\label{eq:V-nat}
\begin{tikzcd}
\CC(c,c')\ar[r,"\FF_{c,c'}"] \ar[d,"\GG_{c,c'}"] & \DD(\FF c,\FF c') \ar[d,"{\DD(1,\varphi_{c'})}"] \\
\DD(\GG c,\GG c') \ar[r,"{\DD(\varphi_{c},1)}"] & \DD(\FF c,\GG c')
\end{tikzcd}   
\end{equation}
is commutative.

If $\CC$ is small, we can define the category of functors $[\CC,\DD]_0,$ whose objects are $\VV$-functors and morphisms are $\VV$-natural transformations. 
This category has a natural $\VV$-enrichment defined by the formula
\begin{equation}
[\CC,\DD](\FF,\GG) = \int_c  \DD(\FF c,\GG c).
\end{equation}
So we obtain a $\VV$-category of $\VV$-functors $[\CC,\DD]$. Here we note that the underlying hom-set $\VV_0(\1, [\CC, \DD](\FF, \GG))$ coincides with the collection of $\VV$-natural transformations $[\CC, \DD]_0(\FF, \GG)$. Moreover, if $\CC$ and $\DD$ are small, there is an equivalence of $\VV$-categories \cite[\S 2.3]{kelly1982basic}
\begin{equation}
[\CC\otimes \DD,\EE] \simeq [\CC,[\DD,\EE]].
\end{equation}

\begin{lemma}\label{lemma:iso}
Let $\varphi:\FF\to \GG$ be a $\VV$-natural transformation such that $\varphi_c$ is an isomorphism for each object $c.$ Then $\varphi$ is an isomorphism.
\end{lemma}
\begin{proof}
We need to check that the collection $(\varphi^{-1}_c:\GG c\to \FF c)$ is a $\VV$-natural transformation. Since $\varphi$ is a $\VV$-natural transformation, we have $\VV(1_{\FF c},\varphi_{c'}) \FF_{c,c'}  =\VV(\varphi_c,1_{\GG c'}) \GG_{c,c'}.$ If we compose it with $\VV(\varphi_{c}^{-1},\varphi_{c'}^{-1}),$ we obtain $\VV(\varphi^{-1}_c,1_{\FF c'}) \FF_{c,c'}  =\VV(1_{\GG c},\varphi_{c'}^{-1}) \GG_{c,c'}.$ Therefore $\varphi^{-1}$ is a $\VV$-natural transformation.
\end{proof}

\subsection{Weighted limits} Let $\CC$ be a small $\VV$-category and let $W,\FF:\CC\to \VV$ be $\VV$-functors. Then the $W$-weighted (or indexed) limit of $\FF$ can be defined as \cite[(3.7)]{kelly1982basic} 
\begin{equation}\label{eq:Wlimit_V}
\{ W,\FF\} = [\CC,\VV](W,\FF).
\end{equation}
So, in this case it exists. Moreover, in this case we see that it is $\VV$-natural in $W$ and $\FF$ 
\begin{equation}
\{,\}:[\CC,\VV] \otimes [\CC,\VV] \longrightarrow \VV. 
\end{equation}

More generally, if $W:\CC\to \VV$ and $\FF:\CC\to \DD$ are $\VV$-functors, then the weighted limit $\{W,\FF\}$ is defined as an object of $\DD$ together with an isomorphism 
\begin{equation}\label{eq:Wlimit_D}
\DD(d,\{W,\FF\}) \cong 
\{ W,\DD(d,\FF-)\},
\end{equation}
which is $\VV$-natural in $d.$ If we take $d=\{W,\FF\},$ then the image of $1_{\{W,\FF\}}$ under the isomorphism is denoted by 
\begin{equation}
\mu_{W,\FF} : W \longrightarrow \DD(\{W,\FF\},\FF-)  
\end{equation}
and called the unit of the weighted limit $\{W,\FF\}$. The isomorphism \eqref{eq:Wlimit_D} is uniquely defined by $\mu_{W,\FF} \in [\CC, \VV]_0(W, \DD(\{W, \FF\}, \FF-)).$ 

A $\VV$-category $\DD$ is called $\VV$-complete, if all weighted limits of functors from a small $\VV$-category $\CC$ exist. In this case the definition of a weighted limit can be uniquely extended to a $\VV$-functor \cite[(3.11) and \S 1.10]{kelly1982basic}
\begin{equation}
\{,\}:[\CC,\VV]\otimes [\CC,\DD] \longrightarrow \DD
\end{equation}
such that the isomorphism \eqref{eq:Wlimit_D} is $\VV$-natural in $W$ and $\FF.$ 

A $\VV$-functor $\Phi:\DD\to \DD'$ is said to  preserve the limit $\{W,\FF\}$ if there is an isomorphism $
\theta:\Phi\{W,\FF\}\xrightarrow{\cong}\{W,\Phi\FF\}$ such that the diagram 
\begin{equation}\label{eq:preservation}
\begin{tikzcd}
W\ar[r,"\mu_{W,\FF}"] 
\ar[d,"\mu_{W,\Phi\FF}"]& 
\DD(\{W,\FF\},\FF-)
\ar[d,"\Phi"]
\\ 
\DD'(\{W,\Phi\FF\},\Phi\FF-) \ar[r, "\DD'(\theta\textit{,} 1)"]
&
\DD'(\Phi\{W,\FF\},\Phi\FF-)
\end{tikzcd}
\end{equation}
is commutative \cite[\S 3.2]{kelly1982basic}. 

A weighted colimit $W\!\star\FF$ in $\DD$ can be defined as a weighted limit $\{W,\FF^{op}\}$ in $\DD^{op}.$ More precisely, if we have functors $W:\CC^{op}\to \VV$ and $\FF:\CC\to \DD,$ then the weighted colimit is an object $W\!\star \FF$ of $\DD$ together with an isomorphism
\begin{equation}\label{eq:W-colimit_def}
\DD(W\!\star \FF,d)\cong \{ W,\DD(\FF-,d)\},
\end{equation}
which is $\VV$-natural in $d.$ If we take $d=W\star \FF,$ then the image of $1_{W\star \FF}$ under the isomorphism is denoted by 
\begin{equation}
\nu_{W,\FF} : W \longrightarrow \DD(\{W,\FF\},\FF-)  
\end{equation}
and called the counit of the weighted colimit $W\star \FF$. The preservation of weighted colimits is defined dually to the preservation of weighted limits.

The category $\DD$ is called $\VV$-cocomplete, if all weighted colimits from small categories $\CC$ exist. In this case the definition of a weighted colimit can be uniquely extended to a  $\VV$-functor \cite[(3.11) and \S 1.10]{kelly1982basic}
\begin{equation}
\star : [\CC^{op},\VV]\otimes [\CC,\DD] \longrightarrow \DD
\end{equation}
such that the isomorphism \eqref{eq:W-colimit_def} is $\VV$-natural in $W$ and $\FF.$ 

A $\VV$-category $\DD$ is called $\VV$-bicomplete, if it is $\VV$-complete and $\VV$-cocomplete. For example, $\VV$ itself is $\VV$-bicomplete \cite[(3.74)]{kelly1982basic}. Moreover, if $\EE$ is a small $\VV$-category, the functor category $[\EE,\VV]$ is $\VV$-bicomplete \cite[Prop. 3.75]{kelly1982basic}. In this case for each object $e$ we can consider the evaluation $\VV$-functor  \cite[(2.17)]{kelly1982basic} 
\begin{equation}\label{eq:evaluation1}
E_e : [\EE,\VV] \longrightarrow \VV, \hspace{1cm} T\mapsto Te.
\end{equation}
This $\VV$-functor preserves all weighted limits and colimits \cite[(3.16)]{kelly1982basic}. In particular,  for a functor $\FF:\CC\to [\EE,\VV]$ we have isomorphisms 
\begin{equation}\label{eq:functor_limit}
\{W,\FF\} e \cong \{W, (\FF -)e  \}, \hspace{1cm} (W\!\star\FF)e\cong W\!\star ((\FF -)e), 
\end{equation}
which are $\VV$-natural in $e,$ $\FF$ and $W$.

For any object $c$ of $\CC$ we can consider $\VV$-functors $\CC(c,-):\CC\to \VV$ and $\CC(-,c):\CC^{op}\to \VV.$ Then the enriched Yoneda lemma says that for any functor $\FF:\CC\to \VV$ there are isomorphisms
\begin{equation}
\{\CC(c,-),\FF\} \cong \FF c, \hspace{1cm} \CC(-,c)\star \FF \cong \FF c    
\end{equation}
$\VV$-natural in $c$ and $\FF$ \cite[(3.12)]{kelly1982basic}.

For a functor $\FF:\CC^{op}\otimes \CC\to \DD$ its end and coend can be presented as weighted limits and colimits \cite[(3.59),(3.66)]{kelly1982basic}
\begin{equation} 
\int_\CC \FF = \{ \Hom_\CC,\FF\}, \hspace{1cm} \int^\CC\!\! \FF = \Hom_\CC \! \star \FF.
\end{equation}
Therefore, if $\DD$ is $\VV$-bicomplete and $\CC$ is small, the ends and coends of functors $\FF:\CC^{op}\otimes \CC\to \DD$  exist. 

\subsection{Conical limits}

For a set $S$ and an object $v$ of $\VV$ we denote by $S\cdot v$ the coproduct of copies of $v$ indexed by $S.$ Let $\LL$ be a small ordinary category (over the category of sets). We denote by $\LL_\VV$ the free $\VV$-category generated by $\LL$ \cite[\S 2.5]{kelly1982basic}. Its objects are objects of $\LL,$ and hom-objects are $\LL_\VV(l,l')=\LL(l,l')\cdot \1.$ If $\CC$ is a $\VV$-category, a functor $\FF:\LL\to \CC_0$ uniquely defines a $\VV$-functor $\overline{\FF}:\LL_\VV \to \CC.$ In particular, the functor $\Delta\1 :\LL\to \VV_0$ that sends all objects to $\1$ and all morphisms to $1_\1$ defines a functor $\overline{\Delta\1}:\LL_\VV \to \VV.$ The conical limit of $\FF$ is defined by 
\begin{equation}
{\lim}_\CC \: \FF = \{\overline{\Delta\1},\overline{\FF}\}.
\end{equation}
The unit \eqref{eq:Wlimit_D}   in this case $\mu_{\overline{\Delta\1},\overline{\FF}}$ is defined by the cone 
\begin{equation}
\eta_{\FF,l}: {\rm lim}_\CC \: \FF \longrightarrow \FF l,
\end{equation}
where $\eta_{\FF,l}=\mu_{\overline{\Delta\1},\overline{\FF},l}.$

If $\CC=\VV,$ then the ordinary limit of the functor $\FF:\LL\to \VV_0$ is always its conical colimit  \cite[\S 3.8]{kelly1982basic} 
\begin{equation}\label{eq:conic-limits-V}
{\lim}_{\VV_0}\: \FF \cong  {\lim}_\VV \:\FF.
\end{equation}
So we can say that ordinary limits in $\VV$ also have a stronger property of being conical limits in the enriched sense. Then the defining property of the weighted limit of a functor $\FF:\LL\to \CC_0$   \eqref{eq:Wlimit_D} can be rewritten as the existence of the following $\VV$-natural isomorphism.
\begin{equation}\label{eq:def-conic-lim}
\CC(c,{\lim}_\CC\:\FF) \cong {\lim}_{\VV} \:  \CC(c, \FF-) 
\end{equation}
If the conical limit exists in 
$\CC$, then 
the ordinary limit  also exists in $\CC_0$ 
and it coincides
with the conical limit \cite[(3.53)]{kelly1982basic}.
\begin{equation}\label{eq:conic-ordinary}
{\lim}_\CC \: \FF \cong {\lim}_{\CC_0}\: \FF
\end{equation}
However, the property of being a conical limit in $\CC$ is in general stronger than being an ordinary limit in $\CC_0.$ For example, the product of two objects $c_1\times c_2$ in $\CC_0$ is an object with a natural isomorphism 
$\CC_0(c,c_1\times c_2)\cong \CC_0(c,c_1)\times \CC_0(c,c_2).$
But the corresponding conical limit in $\CC$ is an object with 
a $\VV$-natural isomorphism 
$\CC(c,c_1\times c_2)\cong \CC(c,c_1)\times \CC(c,c_2).$ On the other hand, if $\VV=\CC$ then for the ordinary product $v_1\times v_2$ we have the isomorphism $\VV(v,v_1\times v_2)\cong \VV(v,v_1)\times \VV(v,v_2).$

Conical colimits are defined dually, and their defining property is 
\begin{equation}\label{eq:def-conic-colim}
\CC( {\rm colim}_\CC\: \FF,c) \cong  {\lim}_{\VV}\: \CC(\FF-,c).
\end{equation}
If $\VV=\CC,$ then the ordinary colimits in $\VV_0$ are conical colimits in $\VV.$
\begin{equation}\label{eq:conic-colimits-V}
{\rm colim}_{\VV_0} \FF \cong   {\rm colim}_\VV \FF. 
\end{equation}
In any $\VV$-category $\CC$ a conical colimit is also an ordinary colimit in $\CC_0.$

If $\Phi:\CC\to \CC'$ is a functor that preserves the conical limit ${\rm lim}_\CC\: \FF,$ then there is an isomorphism $\theta : \Phi({\rm lim}_\CC\:\FF ) \xrightarrow{\cong} {\rm lim}_\CC\:\Phi\FF$ such that the diagram \eqref{eq:preservation} is commutative. It follows that the diagram 
\begin{equation}\label{eq:preservation_conic}
\begin{tikzcd}
\Phi({\rm lim}_\CC\: \FF)\ar[rr,"\theta"] \ar[rd,"\Phi \eta_{\FF,l}"']  & & {\rm lim}_\CC\: \Phi\FF \ar[dl,"\eta_{\Phi \FF,l}"] \\
& \Phi\FF l  &    
\end{tikzcd}
\end{equation}
is also commutative. A dual digramm is commutative for a functor that preserves a conical colimit. 

Since, conical limits and colimits are particular cases of weighted limits and colimits, they exist in any $\VV$-bicomplete category $\CC.$ In particular, if $\DD$ is a small category, then conical limits and colimits  exist in $[\DD,\VV]$ and the evaluation $\VV$-functor
\begin{equation}\label{eq:evaluation_conic}
E_d : [\DD,\VV]\longrightarrow \VV
\end{equation}
preserves them \eqref{eq:evaluation1}. In particular, for any functor $\FF:\LL\to [\DD,\VV]_0$ there are isomorphisms 
\begin{equation}\label{eq:lim-functos}
({\lim}_{[\DD,\VV]}\: \FF)d \cong  {\lim}_{\VV}\: (\FF -)d, \hspace{1cm} ({\rm colim}_{[\DD,\VV]}\: \FF)d \cong  {\rm colim}_{\VV} \: (\FF -)d,
\end{equation}
which are $\VV$-natural in $d$ and natural in $\FF.$

\subsection{Kan extensions}
Let $\CC$ and $\DD$ be small $\VV$-categories and let $\EE$ be a complete $\VV$-category. Consider the functor categories $[\CC,\EE]$ and $[\DD,\EE].$ Then any $\VV$-functor $\FF:\CC\to \DD$ defines a composition $\VV$-functor.
\begin{equation}\label{eq:[F,1]}
[\FF,1]:[\DD,\EE] \longrightarrow [\CC,\EE]
\end{equation}
Moreover, there are two functors 
\begin{equation}\label{eq:Lan-Ran}
{\rm Lan}_\FF, {\rm Ran}_\FF : [\CC,\EE] \longrightarrow [\DD,\EE]
\end{equation}
which are adjoint to $[\FF,1]$ 
\begin{equation}
 {\rm Lan}_\FF  \dashv  [\FF,1] \dashv {\rm Ran}_\FF
\end{equation}
in the following enriched sense \cite[Th. 4.38, Th. 4.50]{kelly1982basic}.
\begin{equation}
[\DD,\EE]({\rm Lan}_\FF S,T) \cong [\CC,\EE](S,T\FF)
\end{equation}
\begin{equation}
[\CC,\EE](T\FF,S) \cong [\DD,\EE](T,{\rm Ran}_\FF S)
\end{equation}
The adjunction in the enriched sense implies the adjunction in the ordinary sense.
\begin{equation}
[\DD,\EE]_0({\rm Lan}_\FF S,T) \cong [\CC,\EE]_0(S,T\FF)
\end{equation}
\begin{equation}
[\CC,\EE]_0(T\FF,S) \cong [\DD,\EE]_0(T,{\rm Ran}_\FF S)
\end{equation}
The functor ${\rm Lan}_\FF S$ is called the left Kan extension of $S$ with respect to $\FF,$ and ${\rm Ran}_\FF S$ is called the right Kan extension of $S$ with respect to $\FF.$ They can be computed on objects as follows \cite[(4.18),(4.9)]{kelly1982basic}.
\begin{equation}
({\rm Lan}_\FF S)d \cong 
\DD(\FF-,d)\star S
 \hspace{1cm}
({\rm Ran}_\FF S)d \cong 
\{\DD(d,\FF -),S\}
\end{equation}
Moreover, if $\EE=\VV,$ then they can be presented in terms of ends and coends \cite[(4.24),(4.25)]{kelly1982basic}.
\begin{equation}\label{eq:LanRan}
{\rm Lan}_\FF S = \int^c \DD(\FF c, - ) \otimes S c, \hspace{1cm}
{\rm Ran}_\FF S = \int_c \VV(\DD(-,\FF c),S c).
\end{equation}

\subsection{Cosmos change}\label{subsec:cosmos_change}

Let $\VV$ and $\WW$ be two B\'enabou cosmoi. A lax monoidal functor is a functor
$\Sigma:\VV\to \WW$ together with a morphism
\begin{equation}
\1_\WW \longrightarrow \Sigma \1_\VV
\end{equation}
and a natural transformation
\begin{equation}\label{eq:nat-monoidal-functor}
\Sigma v\otimes \Sigma v' \longrightarrow \Sigma(v\otimes v')
\end{equation}
satisfying some axioms \cite[p.473]{eilenberg1966closed}. The natural transformation defines a natural transformation 
\begin{equation}
 \Sigma(\VV(v,v')) \longrightarrow \WW(\Sigma v, \Sigma v')   
\end{equation}
as the composition of the morhism  
$\Sigma(\VV(v,v')) \to \WW(\Sigma v,\Sigma(v\otimes \Sigma(v\otimes \VV(v,v'))))$ adjoint to \eqref{eq:nat-monoidal-functor} and a morphism $\WW(\Sigma v,\Sigma(v\otimes \Sigma(v\otimes \VV(v,v')))) \to \WW(\Sigma v, \Sigma v')$ induced by the counit $v\otimes \VV(v,v') \to v'.$ This morphism defines a structure of a closed functor on $\Sigma$ (see \cite[Ch.I,\S 3 and Ch.II,Prop.4.3]{eilenberg1966closed}).

For a lax monoidal functor  $\Sigma:\VV\to \WW$ and a $\VV$-category $\CC$ one can define a $\WW$-category $\Sigma \CC$ (see \cite[p.150-151]{lawvere1973metric} and \cite[Ch.I, Prop.6.1]{eilenberg1966closed}). Its objects are objects of $\CC,$ hom-objects are defined by $\Sigma\CC(c,c') =\Sigma(\CC(c,c')),$ the composition in $\Sigma\CC$ is defined as the composition
\begin{equation}
\Sigma\CC(c',c'')\otimes \Sigma\CC(c,c') \longrightarrow \Sigma( \CC(c',c'')\otimes \CC(c,c') ) \longrightarrow \Sigma\CC(c,c''), 
\end{equation}
while the identity morphism in $\Sigma\CC$ is defined as the composition 
\begin{equation}
\1_\WW \longrightarrow \Sigma \1_\VV \longrightarrow \Sigma(\CC(v,v)).
\end{equation}

Any $\VV$-functor $\FF:\CC \to \DD$ defines a $\WW$-functor $\Sigma\FF:\Sigma\CC\to \Sigma\DD$ such that $(\Sigma\FF)_{c,c'}=\Sigma(\FF_{c,c'}).$

\subsection{Nerve of an augmented category}\label{subsection:nerve}
For ordinary small category there is a notion of nerve, which is a simplicial set associated with a category. For enriched categories it is more complicated. One option is to assume that the underlying monoidal category is semicartesian. In this case the construction of nerve is well defined. This happens because any category enriched over a semicartesian monoidal category has a unique augmentation. If we want to consider the nerve of a category enriched over any monoidal category, we need to equip it by an augmentation.  

In this subsection we assume that $\VV$ is cocomplete symmetric monoidal category (not necessarily B\'enabou cosmos). Denote by $\pt$ the $\VV$-category with one object $*$ and hom given by $\pt(*,*)=\1.$ The composition and identity of $\pt$ are defined as the standard isomorphism $\1\otimes \1\cong \1$ and the identity morphism $\1\to \1.$ An \emph{augmented $\VV$-category} is a small non-empty $\VV$-category 
$\CC$ together with a
 $\VV$-functor 
$\varepsilon:\CC\to \pt$ called augmentation. 

The nerve of an augmented $\VV$-category $\CC$ is a simplicial object $\NN_\bullet(\CC)$ in $\VV,$ whose components are defined as
\begin{equation}
\NN_n(\CC)  = \coprod_{c_0,\dots,c_n}  \CC(c_0,c_1)\otimes \dots \otimes \CC(c_{n-1},c_n).
\end{equation}
In particular, for $n=0$ we have $\NN_0(\CC)=\coprod_{c_0\in {\rm Ob}(\CC)} \1.$
The face maps $d_i:\NN_n(\CC)\to \NN_{n-1}(\CC)$ for $0\leq i\leq n$ are defined as follows. For $0<i<n$ they are induced by  compositions
\begin{equation}
\CC(c_{i-1},c_i)\otimes \CC(c_i,c_{i+1}) \overset{\cong}\longrightarrow \CC(c_i,c_{i+1}) \otimes \CC(c_{i-1},c_i) \overset{\circ}\longrightarrow \CC(c_{i-1},c_{i+1}).
\end{equation}
The morphisms $d_0$ and $d_n$ are induced by the augmentation $\varepsilon:\CC(c_0,c_1)\to \1$ and  $\varepsilon:\CC(c_{n-1},c_n) \to \1.$ The degeneracy maps $s_i:\NN_n(\CC)\to \NN_{n+1}(\CC)$ are defined induced by the identity map $\1\to \CC(c_i,c_i).$ We leave the verification of simplicial relations to the reader.

\section{Modules over an enriched category}

\subsection{Abelian B\'enabou cosmoi} A B\'enabou cosmos $\AA$ is called abelian if  $\AA_0$ is an abelian category satisfying Grothendieck axioms ${\rm (AB4)}$ and ${\rm (AB4^*)}.$ The Grothendieck axioms say that a small (co)product of short exact sequences is a short exact sequence. Further in this section we assume that $\AA$ is an abelian B\'enabou cosmos. 

\begin{proposition}\label{prop:additive-functors}
For any object $a$ of $\AA$ we have the following.
\begin{enumerate}
\item The functor $a\otimes -$ is an additive right exact functor. Moreover, it commutes with small colimits.
\item The functors $\AA(a,-)$ and $\AA(-,a)$ are additive left exact functors. Moreover, $\AA(a-)$ commutes with limits and $\AA(-,a)$ sends colimits to limits.
\end{enumerate}
\end{proposition}
\begin{proof} Recall that direct sums are both products and coproducts. (Co)kernels are (co)equalisers. An additive functor is  left (resp. right) exact if and only if it commutes with kernels (resp. cokernels). So we just need to prove that $a\otimes -$ commutes with small colimits,  $\AA(a-)$ commutes with small limits and $\AA(-,a)$ sends colimits to limits. The fact that $a\otimes -$ commutes with small colimits follows from the fact that it is left adjoint. Similarly we prove that $\AA(a-)$ commutes with limits.  Let us prove that $\AA(-,a)$ sends colimits to limits.  Note that small colimits in $\AA$ are conical colimits in the enriched sense \eqref{eq:conic-colimits-V}. Then the result follows from the formula \eqref{eq:def-conic-colim}. 
\end{proof}

\subsection{Category of modules} Let $\CC$ be a small $\AA$-category. A $\CC$-module $M$ is a $\AA$-functor $M:\CC\to \AA.$ A morphism of modules is a $\AA$-natural transformation. The category of $\CC$-modules is denoted by 
\begin{equation}
\Mod(\CC) = [\CC,\AA]_0.
\end{equation} 
A morphism in this category $f:M\to N$ is a $\AA$-natural transformation  \eqref{eq:V-nat}. The set of morphisms will be denoted by 
\begin{equation}
\Hom_\CC(M,N) = [\CC,\AA]_0(M,N).
\end{equation}
For an object $c$ of $\CC$ we consider the evaluation functor
\begin{equation}\label{eq:eval-mod}
{\rm Ev}_c : {\rm Mod}(\CC) \longrightarrow \AA, \hspace{1cm} M\mapsto Mc. 
\end{equation}

\begin{proposition}\label{prop:limits_of_modules}
The category of $\CC$-modules ${\rm Mod}(\CC)$ is bicomplete. Moreover, for any object $c$ the evaluation functor ${\rm Ev}_c:{\rm Mod}(\CC) \to \AA$ preserves limits and colimits. In particular, for any functor from a small category $\FF:\LL\to {\rm Mod}(\CC)$ there  isomorphisms \begin{equation} 
({\rm lim}\: \FF)c \cong  {\rm lim}\: ((\FF-) c), \hspace{1cm} ({\rm colim}\: \FF)c \cong  {\rm colim}\: ((\FF-) c),
\end{equation}
which are natural in $c$ and $\FF,$ and the following diagrams are commutative.
\begin{equation}
\begin{tikzcd}[column sep={13mm,between origins}]
({\rm lim}\: \FF)c \ar[rr,"\cong"]\ar[dr,"\eta_{\FF,l,c}"'] && {\rm lim}\: ((\FF-) c) \ar[dl,"\eta_{(\FF-)c,l}"] \\
& (\FF l)c
\end{tikzcd}
\begin{tikzcd}[column sep={13mm,between origins}]
& (\FF l)c \ar[dl,"\epsilon_{\FF,l,c}"'] \ar[dr,"\epsilon_{(\FF-)c,l}"] 
\\    
({\rm colim}\: \FF)c \ar[rr,"\cong"] && {\rm colim}\: ((\FF-) c)
\end{tikzcd}
\end{equation}
\end{proposition}
\begin{proof} 
It follows from the fact that $[\CC,\AA]$ is $\AA$-complete,  that the evaluation functor $E_c:[\CC,\AA]\to \AA$ preserves conical limits and colimits 
\eqref{eq:evaluation_conic}, 
\eqref{eq:preservation_conic}, 
and that conical (co)limits are ordinary limits \eqref{eq:conic-ordinary}. 
\end{proof}

\begin{theorem} For a small $\AA$-category $\CC$ 
the category of $\CC$-modules $\Mod(\CC)$ is an abelian category satisfying Grothendieck axioms ${\rm (AB4)}$ and ${\rm (AB4^*)}.$  
\end{theorem}
\begin{proof}
Let $\varphi,\psi:M\to N$ be morphisms of $\CC$-modules. Consider a collection $\varphi - \psi = (\varphi_c- \psi_c)$ indexed by objects of $\CC.$ Let us check that $\varphi - \psi$ is a morphism of $\CC$-modules using the definition \eqref{eq:V-nat}. Since $\AA(Mc,-)$ and $\AA(-,Nc')$ are additive functors (Proposition \ref{prop:additive-functors}), they commute with differences of morphisms. Combining this with the fact that composition in $\AA$ is bilinear, we obtain that $\varphi - \psi$ is a morphism of $\CC$-modules. Therefore  $\Hom_\CC(M,N) \subseteq \prod_c  \AA_0(Mc,Nc)$ is a subgroup. In particular, $\Hom_\CC(M,N)$ has a natural structure of a group. Using that the composition in $\Mod(\CC)$ is defined component-wise and the composition in $\AA_0$ is bilinear, we obtain that it is bilinear in $\Mod(\CC).$ Therefore, $\Mod(\CC)$ is a preadditive category. 

Proposition \ref{prop:limits_of_modules} implies that $\Mod(\CC)$ has finite products and coproducts. Consider a zero $\CC$-module $0$ defined by $0(c)=0.$ Using the definition of a morphism of $\CC$-modules \eqref{eq:V-nat}, it is easy to check that $0$ is a zero object in $\Mod(\CC).$ Since any preadditive category with finite products is additive \cite[\S 2.1]{popescu1973abelian}, we obtain that $\Mod(\CC)$ is an additive category. Moreover, since $\Mod(\CC)$ is bicomplete, we obtain that it is also preabelian. 

In order to prove that $\Mod(\CC)$ is abelian, we need to prove that for any morphism $\varphi:M\to N$ the morphism ${\rm Coim}(\varphi)\to {\rm Im}(\varphi)$ is an isomorphism. Proposition \ref{prop:limits_of_modules} imply that there are isomorphisms ${\rm Coim}(\varphi)c \cong  {\rm Coim}(\varphi_c)$ and ${\rm Im}(\varphi)c\cong {\rm Im}(\varphi_c),$ which are natural in the sense that the diagram 
\begin{equation}
\begin{tikzcd}
Mc \ar[r] \ar[d,equal] & {\rm Coim}(\varphi)c \ar[r] \ar[d,"\cong"] & {\rm Im}(\varphi)c \ar[r] \ar[d,"\cong"] & Nc \ar[d,equal] \\
Mc \ar[r] & {\rm Coim}(\varphi_c) \ar[r] & {\rm Im}(\varphi_c) \ar[r] & Nc
\end{tikzcd}
\end{equation}
is commutative. 
Using that the morphism ${\rm Coim}(\varphi_c)\to {\rm Im}(\varphi_c)$ is an isomorphism in $\AA,$ we obtain that ${\rm Coim}(\varphi)c\to {\rm Im}(\varphi)c$ is an isomorphism. By Lemma \ref{lemma:iso} we obtain that the morphism ${\rm Coim}(\varphi)c\to {\rm Im}(\varphi)c$ is an isomorphism. Therefore $\Mod(\CC)$ is abelian. 

Let us finally prove that ${\rm Mod}(\CC)$ satisfies the axiom ${\rm (AB4)},$ and the axiom ${\rm (AB4^*)}$ can be proved similarly. Consider a collection of short exact sequences $\FF'_i \mono \FF_i \epi \FF''_i.$ We need to prove that the sequence $\coprod_i \FF_i\to \coprod_i \FF_i \to \coprod_i \FF'$ is short exact. Since kernels, cokernels and images are defined component-wise in ${\rm Mod}(\CC)$ we know that $\FF_ic\to \FF_i c \to \FF'_i c$ is short exact for any $i$ and $c$, and we need to prove that $(\coprod_i \FF_i)c\to (\coprod_i \FF_i)c \to (\coprod_i \FF')c$ is short exact.  By Proposition \ref{prop:limits_of_modules} we have natural isomorphisms $(\coprod_i \FF'_i)c\cong \coprod_i (\FF'_ic),$ $(\coprod_i \FF_i)c\cong \coprod_i (\FF_ic)$ and $(\coprod_i \FF''_i)c\cong \coprod_i (\FF''_ic)$ such that the diagram 
\begin{equation}
\begin{tikzcd}
(\coprod_i \FF'_i)c \ar[r] \ar[d,"\cong"] 
& 
(\coprod_i \FF_i)c \ar[r] \ar[d,"\cong"] 
& (\coprod_i \FF''_i)c \ar[d,"\cong"]  
\\
\coprod_i (\FF'_ic) \ar[r] 
& 
\coprod_i (\FF_ic) \ar[r]  
& \coprod_i (\FF''_ic) 
\end{tikzcd}
\end{equation}
is commutative in $\AA$. Combining this with the fact that $\AA$ satisfies ${\rm (AB4)}$ we obtain that the sequence $(\coprod_i \FF_i)c\to (\coprod_i \FF_i)c \to (\coprod_i \FF')c$ is short exact. 
\end{proof}

\subsection{Induced modules}

Let $f:\CC\to \DD$ be a $\AA$-functor between small $\AA$ categories. Then we consider functors 
\begin{equation}
f^* : \Mod(\DD) \longrightarrow \Mod(\CC), \hspace{1cm}
f_!,f_* : \Mod(\CC) \longrightarrow \Mod(\DD)
\end{equation}
defined by the formulas $f^*=[f,1]_0,$ $f_!=({\rm Lan}_f)_0$ and $f_*=({\rm Ran}_f)_0$ (see \eqref{eq:Lan-Ran}). Then there are adjunctions
\begin{equation}\label{eq:f-adj}
f_! \dashv f^* \dashv f_*. 
\end{equation}
The $\CC$-module $f^*N$ will be called the restricted module of $N,$ the $\DD$-module $f_!M$ is called the induced module of $M$, and $f_*M$ is called the coinduced module of $M.$ By \eqref{eq:LanRan} we have the following explicit  formulas
\begin{equation}\label{eq:end_formulas}
f_! M = \int^c \DD(f c, - ) \otimes M c, \hspace{1cm}
f_* M = \int_c \AA(\DD(-,f c),M c).
\end{equation}

If we have two $\AA$-functors $f:\CC\to \DD$ and $g:\DD\to \EE,$ we have \cite[Th. 4.47]{kelly1982basic}
\begin{equation}\label{eq:induced_comp}
(g f)^*= f^* g^*, \hspace{1cm} (g f)_! \cong g_! f_!, \hspace{1cm} (g f)_* \cong g_* f_*. 
\end{equation}

Projective (resp. injective) objects of the abelian category of $\CC$-modules are called projective (resp. injective) $\CC$-modules.

\begin{proposition}\label{prop:coinduced_abelian}
Let $f:\CC\to \DD$ be a $\AA$-functor between small $\AA$-categories. Then the following holds. 
\begin{enumerate}
    \item $f^*$ is an exact functor, $f_!$ is right exact functor, $f_*$ is left exact functor.
    \item $f_!$ sends projective $\CC$-modules to projective $\DD$-modules. 
    \item $f_*$ sends injective $\CC$-modules to injective $\DD$-modules.
\end{enumerate}
\end{proposition}
\begin{proof}
(1). It follows from  \eqref{eq:f-adj} and the fact that left adjoint functor is left exact and right adjoint functor is left exact.

(2). Let $P$ be a projective $\CC$-module. Then the functor $\Hom_{\CC}(P,-)$ is exact. Since $f^*$ is also exact, we obtain that $\Hom_{\DD}(f_!P,-)=\Hom_{\CC}(P,f^*(-))$ is exact. Hence $f_!P$ is projective. 

(3). Similar. 
\end{proof}

\begin{proposition}\label{prop:surjective_on_objects}
If $f:\CC\to \DD$ is surjective on objects, then $f^*$ is a faithful functor. Moreover, in this case for any $\DD$-module $M$ the counit of adjunction $f_!f^*M\to M$ is an epimorphism, and the unit of adjunction $M\to f_*f^* M$ is a monomorphism. 
\end{proposition}
\begin{proof}
For two morphisms $\varphi,\psi:M\to N$ of $\DD$-modules the equation $f^*(\varphi)=f^*(\psi)$ implies  $\varphi_{fc} = \psi_{fc}$ for any object $c$ of $\CC.$ Since $f$ is surjective on objects the equation $\varphi_{fc} = \psi_{fc}$ for each $c$ implies $\varphi_{d} = \psi_{d}$ for each $d.$ It follows that $f^*$ is faithfull. 

Using that $f_!$ is left adjoint to $f^*,$ we obtain that the composition
\begin{equation}
f^* \xrightarrow{u f^*} f^*f_!f^* \xrightarrow{f^* e } f^*    
\end{equation}
is identical, where $u$ is the unit of the adjunction, and $e$ is the counit of the adjunction. Therefore $f^* (e_M) : f^*(f_!f^* M) \to f^*(M)$ is an epimorphism. Since $f^*$ is faithfull, we obtain that $e_M : f_!f^* M \to M$ is also epimorphism. Similarly we prove that $M\to f_*f^*M$ is a monomorphism. 
\end{proof}

\subsection{Free modules} Any set $X$ can be treated as a \emph{discrete  $\AA$-category}, whose objects are elements of $X,$ and the hom-sets are defined so that $X(x,x)=\1$ and $X(x,x')=0$ for $x\ne x'.$ Then $\Mod(X)$ is isomorphic to the category of families $(a_x)_{x\in X}$ of objects from $\AA.$ We will identify $X$-modules with these families.

For any small $\AA$-category $\CC$ we treat ${\rm Ob}(\CC)$ as a discrete $\AA$-category. Then we have the canonical embedding of $\AA$-categories $\iota:{\rm Ob}(\CC)\to \CC.$ So we can consider the functors
\begin{equation}
\Forg:\Mod(\CC) \longrightarrow \Mod({\rm Ob}(\CC))
, \hspace{1cm}
\Fr, \CoFr : \Mod({\rm Ob}(\CC)) \longrightarrow \Mod(\CC)
\end{equation}
defined as
$\Forg=\iota^*,$ $\Fr=\iota_!,$  
$\CoFr = \iota_*.$
So the forgetful functor sends a $\CC$-module $M$ to the family $(M_c)_{c\in {\rm Ob}(\CC)}$ and $\Fr$ and $\CoFr$ are left and right adjoint functors to the forgetful functor
\begin{equation}\label{eq:adj_free}
\Fr \dashv \Forg \dashv \CoFr.
\end{equation}
Equations \eqref{eq:end_formulas} imply there are the following explicit formulas 
\begin{equation}\label{eq:fre_explicit}
\Fr(a_*) \cong \coprod_c a_c\otimes  \CC(c,-), \hspace{1cm} \CoFr(a_*) \cong  \prod_c \AA(\CC(-,c),a_c). 
\end{equation}
A $\CC$-module of the form $\Fr(a_*)$ is called a free module, and the $\CC$-module of the form $\CoFr(a_*)$ is called cofree $\CC$-module.

Proposition \ref{prop:surjective_on_objects} implies that  for any $\CC$-module $M$  we have an epimorphism from the free $\CC$-module, and a monomorphism into a cofree $\CC$-module
\begin{equation}\label{eq:fr_forg_epi}
\Fr(\Forg(M))\epi M, \hspace{1cm}   M\mono \CoFr(\Forg(M)).
\end{equation}

\begin{proposition}\label{prop:projectives}  Assume that $\AA$ has enough of projectives. Then $\Mod(\CC)$ has enough of projectives. Moreover, a $\CC$-module is projective if and only if it is a direct summand  of a module of the form $\Fr(p_*)$ such that $p_c$ is a projective object of $\AA$ for any $c.$ 
\end{proposition}
\begin{proof}
Consider a $\CC$-module $M$ and apply the forgetful functor $\Forg(M)=(M_c)_{c\in {\rm Ob}(\CC)}.$ Then for each $c$ there exists an epimorphism from a projective module $p_c\epi M_c.$ Therefore, we have an epimorphism in the category of ${\rm Ob}(\CC)$-modules $p_*\epi \Forg(M).$ Since $\Fr=\iota_!$ is right exact, the morphism $\Fr(p_*) \epi \Fr(\Forg(M))$ is an epimorphism. Composing with the epimorphism \eqref{eq:fr_forg_epi}, we obtain an epimorphism $\Fr(p_*)\epi M.$ Proposition \ref{prop:coinduced_abelian} implies that $\Fr(p_*)$ is projective. Then $\Mod(\CC)$ has enough of projectives. Note that, if $M$ is a projective $\CC$-module, then the epimorphism $\Fr(p_*)\epi M$ splits, and $M$ is a direct summand of $\Fr(p_*).$ 
\end{proof}

Similarly we can prove the following. 

\begin{proposition}\label{prop:injectives}
Assume that $\AA$ has enough of injectives. Then $\Mod(\CC)$ has enough of injectives. Moreover, a $\CC$-module is injective if and only if it is a direct summand of a module $\CoFr(i_*)$ such that $i_c$ is an injective object of $\AA$ for any $c.$
\end{proposition}

\subsection{Flat and projective categories}

We say that an object $a$ of $\AA$ is $\AA$-flat (resp $\AA$-projective), if $a\otimes - $ (resp. $\AA(a,-)$) is an exact functor. We say that a $\AA$-category is $\AA$-flat (resp. $\AA$-projective), if $\CC(c,c')$ is $\AA$-flat (resp. $\AA$-projective) for any objects $c,c'.$

\begin{lemma}\label{lemma:Fr_exact}
If $\CC$ is $\AA$-flat (resp. $\AA$-projective), then ${\rm Fr}$ (resp. ${\rm CoFr}$) is an exact functor.
\end{lemma}
\begin{proof}
 It follows from the formulas  \eqref{eq:fre_explicit} and axioms ${\rm (AB4)},$ ${\rm (AB4^*)}.$  
\end{proof}

\begin{lemma}\label{lemma:sets}
Let $X$ and $Y$ be sets and $f:X\to Y$ be a map. If we treat $X$ and $Y$ as discrete $\AA$-categories, and $f$ as an $\AA$-functor, then the functors $f_!,f_*:\Mod(X)\to \Mod(Y)$ can be described by the formulas 
\begin{equation}
(f_!(a_*))_y = \coprod_{x\in f^{-1}(y)} a_x, \hspace{1cm} (f_*(a_*))_y=\prod_{x\in f^{-1}(y)} a_x. 
\end{equation}
\end{lemma}
\begin{proof}
It follows from \eqref{eq:end_formulas}. 
\end{proof}

\begin{proposition}\label{prop:acyclic} 
Let $f:\CC\to \DD$ be a $\AA$-functor between small $\AA$-flat $\AA$-categories. Assume that $\AA$ has enough of projectives. Then for any ${\rm Ob}(\CC)$-module $a_*$ the $\CC$-module $\Fr(a_*)$ is $f_!$-acyclic.
\end{proposition}
\begin{proof}
Denote by $(p_*)_\bullet$ a projective resolution of $a_*$ in $\Mod({\rm Ob}(\CC)).$ Using Lemma \ref{lemma:Fr_exact} and Proposition \ref{prop:projectives} we obtain that $\Fr((p_*)_\bullet)$ is a projective resolution of $\Fr(a_*).$ Then $L_nf_!(\Fr(a_*))$ is equal to the homology of the chain complex $f_!(\Fr((p_*)_\bullet)).$ The commutativity of the diagram 
\begin{equation}
\begin{tikzcd}
{\rm Ob}(\CC) \ar[r,"\iota"] \ar[d,"{\rm Ob}(f)"] & \CC \ar[d,"f"] \\
{\rm Ob}(\DD)\ar[r,"\iota"] & \DD 
\end{tikzcd}
\end{equation}
and equations \eqref{eq:induced_comp} imply that $f_!\ \Fr = \Fr \ {\rm Ob}(f)_!.$ Lemma \ref{lemma:sets} implies that the functor ${\rm Ob}(f)_!$ is exact. Lemma \ref{lemma:Fr_exact} implies that $\Fr$ is exact.  Therefore $f_!\Fr$ is also exact. Hence $H_n(f_!(\Fr((p_*)_\bullet))) =0$ for $n\neq 0.$ 
\end{proof}

Similarly we obtain the following. 

\begin{proposition} 
Let $f:\CC\to \DD$ be a $\AA$-functor between small $\AA$-projective $\AA$-categories. Assume that $\AA$ has enough of injectives. Then for any ${\rm Ob}(\CC)$-module $a_*$ the $\CC$-module $\CoFr(a_*)$ is $f_*$-acyclic.
\end{proposition}

\subsection{Bar resolution}

For a small $\AA$-category $\CC$ we consider the adjunction \eqref{eq:adj_free}
\begin{equation}
\Fr: \Mod({\sf Ob}(\CC)) \rightleftarrows \Mod(\CC) : \Forg.
\end{equation}
For a ${\sf Ob}(\CC)$-module $a_*=(a_c)_{c\in {\sf Ob}(\CC)}$ we have
$\Fr(a_*) \cong  \coprod_{c} a_c \otimes\CC(c,-).$
The unit of the adjunction $h:a_*\to \Forg(\Fr(a_*))$ is 
defined as the composition
\begin{equation}
a_c\longrightarrow a_c \otimes \1\longrightarrow a_c\otimes  \CC(c,c)\hookrightarrow \coprod_{c'}a_{c'} \otimes \CC(c,c'). 
\end{equation}
The counit $e:\Fr(\Forg(M))\to M$ 
\begin{equation}
e_{c'}:\coprod_c Mc\otimes \CC(c,c') \longrightarrow Mc',
\end{equation}
has components, which are adjoint to  $M_{c,c'}:\CC(c,c')\to \AA(M_c,M_c').$

Consider the associated comonad on the category of $\CC$-modules
\begin{equation}
\overline\Fr:\Mod(\CC)\longrightarrow \Mod(\CC), 
\hspace{1cm}
\overline\Fr(M) = \coprod_c Mc \otimes \CC(c,-).
\end{equation}
The comultiplication $\delta:\overline\Fr \to \overline\Fr^2$ is defined by the whiskering $\delta= \Fr \: h\: \Forg.$

We denote by ${\rm Bar}_\bullet(M)$ the bar construction associated with this comonad \cite[\S 1]{ivanov2021homology}, \cite{appelgate1969homology}. So ${\rm Bar}_\bullet(M)$ is a simplicial $\CC$-module such that 
${\rm Bar}_n(M) = \overline\Fr^{n+1}(M).$ 
By induction we obtain 
\begin{equation}
{\rm Bar}_n(M)=\coprod_{c_0,\dots,c_n} Mc_0 \otimes \CC(c_0,c_1) \otimes \dots \otimes  \CC(c_{n-1},c_n)\otimes \CC(c_n,-).
\end{equation}
The face maps are defined by the formula $d_i=\overline\Fr^i\: e \: \overline\Fr^{n-i},$ and the degeneracy maps are defined by $s_i=\overline\Fr^i\: \delta \: \overline\Fr^{n-i}.$ 

Let us give a more detailed description of the face maps $d_i:{\rm Bar}_n(M)\to {\rm Bar}_{n-1}(M).$ The face map $d_0$ is induced by the evaluation map $Mc_0\otimes \CC(c_0,c_1)\to Mc_1.$ For $0<i<n$ the face map $d_i$ is induced by the map 
\begin{equation}\label{eq:comp}
\CC(c_{i-1},c_{i}) \otimes \CC(c_i,c_{i+1}) \overset{\cong} \longrightarrow \CC(c_{i+1},c_i)\otimes \CC(c_i,c_{i-1}) \overset{\circ}\longrightarrow \CC(c_{i-1},c_{i+1}).
\end{equation}
And similarly, the map $d_n$ is induced by the map 
\begin{equation}\label{eq:comp2}
\CC(c_{i-1},c_{n}) \otimes \CC(c_n,-) \overset{\cong} \longrightarrow \CC(-,c_n)\otimes \CC(c_n,c_{n-1}) \overset{\circ}\longrightarrow \CC(c_{i-1},-).
\end{equation}

For any simplicial object $A_\bullet$ in an abelian category, we denote by $C(A)$ its non-normalised chain complex, and by $N(A)$ its normalised chain complex. In particular, we consider the chain complexes $C({\rm Bar}_\bullet(M))$ and $N({\rm Bar}_\bullet(M)).$ We will also use notation $H_*(A)=H_*(C(A))=H_*(N(A)).$

\begin{proposition}\label{prop:bar_resolution}
Let $f:\CC\to \DD$ be a $\AA$-functor between small $\AA$-flat $\AA$-categories. Assume that $\AA$ has enough of projectives.
Then the chain complex of the bar construction $C({\rm Bar}_\bullet(M))$ is an $f_!$-acyclic resolution of $M.$ In particular, the left derived functors of $f_!$ can be described as
\begin{equation}
L_n f_! (M) \cong H_n(f_!{\rm Bar}_\bullet(M)).
\end{equation}
\end{proposition}
\begin{proof}
By \cite[Lemma 1.2]{ivanov2021homology} the augmented simplicial object $\Forg({\rm Bar}^\bullet(M)) \to \Forg(M)$ has an extra degeneracy map. Therefore the chain complex $C(\Forg({\rm Bar}^\bullet(M)))$ is homotopy equivalent to $\Forg(M).$ Since $\Forg(-)$ is faithful and exact, and $C(-)$ commutes with any additive functor, we obtain that $H_n({\rm Bar}_\bullet(M))=0$ for $n \ne 0$ and $H_0({\rm Bar}_\bullet(M))\cong M.$ The fact that ${\rm Bar}_\bullet(M)$ is $f_!$-acyclic follows from Proposition \ref{prop:acyclic}.
\end{proof}

Dually, we consider the adjunction 
\begin{equation}
\Forg: \Mod(\CC) \rightleftarrows \Mod({\rm Ob}(\CC)): \CoFr,
\end{equation}
where
$\CoFr(a_*) = \prod_c \AA(\CC(-,c),a_c). $
The counit of the adjunction is defined as the composition 
\begin{equation}\label{eq:counit_forg-CoFr}
\prod_{c'} \AA(\CC(c',c),a_c) \epi \AA(\CC(c,c),a_c) \longrightarrow \AA(\1,a_c) \overset{\cong}\longrightarrow a_c. 
\end{equation}
The unit of the adjunction 
\begin{equation}\label{eq:unit_forg-CoFr}
M{c'} \longrightarrow \prod_c \AA(\CC(c',c),Mc)
\end{equation}
is induced by the morphisms adjoint to $M_{c',c}:\CC(c',c)\to \AA(Mc',Mc).$ The corresponding monad is defined by the formula
\begin{equation}\overline\CoFr:\Mod(\CC)\longrightarrow \Mod(\CC),\hspace{1cm}
\overline\CoFr(M) = \prod_c \AA(\CC(-,c),Mc).
\end{equation}
Consider the corresponding cobar cosimplicial endofunctor on $\Mod(\CC)$ such that ${\rm CoBar}^n(M)=\overline\CoFr^{n+1}(M).$ 
Using the isomorphism
\begin{equation}
\AA(\CC(-,c), \AA(\CC(c,c'),Mc') ) \cong \AA(\CC(-,c)\otimes \CC(c,c'),Mc') 
\end{equation}
by induction we obtain 
\begin{equation}
{\rm CoBar}^{n}(M) \cong  \prod_{c_0,\dots,c_n}\AA(\CC(-,c_0)\otimes \dots \otimes \CC(c_{n-1},c_n),Mc_n).
\end{equation}
The coface maps $d^i:{\rm CoBar}^n(M)\to {\rm CoBar}^{n+1}(M)$ for $0\leq i\leq n+1$ are defined as follows. For $0\leq i< n+1$ they are dual to the morphisms induced by the composition \eqref{eq:comp2}, \eqref{eq:comp}. The coface map $d^{n+1}$ is induced by the unit \eqref{eq:unit_forg-CoFr} $Mc_n \to \prod_{c_{n+1}} \AA(\CC(c_n,c_{n+1}),M_{c_{n+1}}).$
The codegenerasy maps $s^i:{\rm CoBar}^n(M)\to {\rm CoBar}^{n-1}(M)$ are induced by the counit \eqref{eq:counit_forg-CoFr}.

\begin{proposition}
Let $f:\CC\to \DD$ be a $\AA$-functor between small $\AA$-projective $\AA$-categories. Assume that $\AA$ has enough of injectives.
Then the chain complex of the cobar construction $C({\rm CoBar}^\bullet(M))$ is an $f_*$-acyclic resolution of $M.$ In particular, the right derived functors of $f_*$ can be described as
\begin{equation}
R^n f_* (M) \cong H^n(f_*{\rm CoBar}_\bullet(M)).
\end{equation}
\end{proposition}
\begin{proof}
Similar to Proposition  \ref{prop:bar_resolution}.
\end{proof}

\subsection{Homology of augmented categories}

Assume that $\CC$ is an augmented $\AA$-category with the augmentation $\varepsilon:\CC\to \pt$ (Subsection \ref{subsection:nerve}). Note that $\Mod(\pt)\cong \AA.$ Therefore the augmentation $\varepsilon$ induces functors of the trivial module, invariants and coinvariants
\begin{equation}
\Triv: \AA \to \Mod(\CC), \hspace{1cm}  \Coinv, \Inv: \Mod(\CC) \to \AA 
\end{equation}
defined by the formulas
\begin{equation}\label{eq:triv_inv_coinv}
\Triv = \varepsilon^*, \hspace{5mm} \Coinv= \varepsilon_!,\hspace{5mm} \Inv=\varepsilon_*.
\end{equation}
Then there are adjunctions
\begin{equation}
    \Coinv \dashv \Triv \dashv \Inv, 
\end{equation}
and Proposition \ref{prop:surjective_on_objects} implies that $\Triv$ is a faithfull functor. 

If we expand the definition of the functor of trivial module, we obtain that for $a\in \AA$ we have 
\begin{equation}\label{eq:trivial}
\Triv(a)(c) = a
\end{equation}
and the morphism $\Triv(a)_{c,c'} $ is defined as the composition 
$\CC(c,c')\to \1 \to \AA(a,a).$ 

By \eqref{eq:end_formulas} the functor of coinvariants can be described as the coequaliser 
\begin{equation}\label{eq:coinv}
\coprod_{c,c'} M c \otimes \CC(c,c')  \rightrightarrows \coprod_c M c  \to \Coinv(M),
\end{equation}
and the functor of invariants as the equaliser
\begin{equation}
\Inv(M) \to \prod_c  M c \rightrightarrows \prod_{c,c'} \VV( \CC(c,c'), M c').
\end{equation}
The commutativity of the diagram
\begin{equation}
\begin{tikzcd}
{\rm Ob}(\CC) \ar[r,"\iota"] \ar[rd] & \CC \ar[d,"\varepsilon"] \\
& \pt
\end{tikzcd} 
\end{equation}
and Lemma \ref{lemma:sets} imply the isomorphisms
\begin{equation}\label{eq:coinv_free}
\Coinv(\Fr(a_*)) \cong  \coprod_c a_c, \hspace{1cm} \Inv(\CoFr(a_*)) \cong  \prod_c a_c.
\end{equation}

\begin{example}
If $\AA={\rm Ab}$ is the category of abelian groups, $G$ is a group and $\CC=\ZZ[G]$ is the group algebra, then $\CC$-modules are $\ZZ[G]$-modules in the usual sense, and and the functor of invariants and coinvariants coincide with the usual functors
\begin{equation}
\Inv(M)= \{m\in M\mid gm=m\},
\end{equation}
\begin{equation}
\Coinv(M) = M/\langle gm-m\mid g\in G, m\in M \rangle.
\end{equation}
\end{example}

Assume that $\AA$ has enough of projectives. We define homology of an augmented $\AA$-category $\CC$ as 
\begin{equation}
H_n(\CC,M) = L_n\Coinv (M).
\end{equation} 
Homology with the trivial coefficients are denoted by 
\begin{equation}
H_n(\CC) = H_n(\CC,\Triv(\1)).
\end{equation}
Similarly, if $\VV$ has enough of injectives, we define cohomology as 
\begin{equation}
H^n(\CC,M)=R^n\Inv(M), \hspace{1cm} H^n(\CC)=H^n(\CC,\Triv(\1)).
\end{equation}

Consider a simplicial object $S_\bullet(\CC,M)$ in $\AA,$ whose components are
\begin{equation}
S_n(\CC,M) = \coprod_{c_0,\dots,c_n} Mc_0\otimes \CC(c_0,c_1)\otimes \dots \otimes \CC(c_{n-1},c_n).
\end{equation}
The face maps $d_i:S_n(\CC)\to S_{n-1}(\CC)$ are defined so that for $0<i<n$ they are induced by  compositions
\begin{equation}\label{eq:d_i}
\CC(c_{i-1},c_i)\otimes \CC(c_i,c_{i+1}) \overset{\cong}\longrightarrow \CC(c_i,c_{i+1}) \otimes \CC(c_{i-1},c_i) \overset{\circ}\longrightarrow \CC(c_{i-1},c_{i+1}).
\end{equation}
The morphism $d_0$ is induced by the map $Mc_0\otimes \CC(c_0,c_1)\to Mc_1$ adjoint to $M_{c_0,c_1}.$ The morphism $d_n$ is induced by the augmentation  $\varepsilon:\CC(c_{n-1},c_n) \to \1.$ The degeneracy maps $s_i:S_n(\CC)\to S_{n+1}(\CC)$ are defined induced by the identity map $\1\to \CC(c_i,c_i).$

\begin{theorem}\label{th:homology_of_augmented}
Let $\CC$ be an augmented $\AA$-flat $\AA$-category. Assume that $\AA$ has enough of projectives. Then for any $\CC$-module $M$ there is an isomorphism
\begin{equation}
H_*(\CC,M) \cong H_*(S_\bullet(\CC,M)).
\end{equation}
\end{theorem}
\begin{proof}
The isomorphism \eqref{eq:coinv_free} implies that $\Coinv(\overline \Fr^{n+1}(M))\cong  
\coprod_c \overline\Fr^n(M)c.$ Therefore
\begin{equation}
\Coinv({\rm Bar}_n(M))\cong   \coprod_{c_0,\dots,c_n} Mc_0\otimes 
\CC(c_0,c_1)
\otimes 
\dots \otimes 
\CC(c_{n-1},c_n).
\end{equation}
Using this we obtain
\begin{equation}\label{eq:coinv_bar}
 \Coinv({\rm Bar}_\bullet(M))\cong S_\bullet(\CC,M).   
\end{equation}
So the statement follows from Proposition \ref{prop:bar_resolution}.
\end{proof}

Consider the nerve $\NN(\CC)$ of $\CC$ (Subsection \ref{subsection:nerve}). Note that 
\begin{equation}\label{eq:NN_S}
\NN_\bullet(\CC)\cong S_\bullet(\CC,{\sf Triv}(\1)).
\end{equation}

\begin{corollary}\label{cor:homology}
Under the assumption of Theorem \ref{th:homology_of_augmented} there is an isomorphism
\begin{equation}
H_*(\CC)\cong H_*(\NN_\bullet(\CC)).
\end{equation}
\end{corollary}

Consider a cosimplicial object $S^\bullet(\CC,M)$ in $\AA,$ whose components are
\begin{equation}
S^n(\CC,M) = \prod_{c_0,\dots,c_n} 
\AA(\CC(c_0,c_1)\otimes \dots \otimes \CC(c_{n-1},c_n),Mc_n).
\end{equation}
The coface maps $d^i:S^n(\CC,M) \to S^{n+1}(\CC,M)$ for $0\leq i\leq n+1$ are described as follows. For $0<i<n+1$ they are dual to the compositions \eqref{eq:d_i}. The morphism $d^0$ is dual to the map induced by the augmentation $\varepsilon : \CC(c_0,c_1)\to \1.$ The coface map $d^{n+1}$ is induced by the unit \eqref{eq:unit_forg-CoFr} $Mc_n \to \prod_{c_{n+1}} \AA(\CC(c_n,c_{n+1}),M_{c_{n+1}}).$
The codegenerasy maps $s^i$ are induced by the counit \eqref{eq:counit_forg-CoFr}.

\begin{theorem}\label{th:cohomology_of_augmented}
Let $\CC$ be an augmented $\AA$-projective $\AA$-category. Assume that  $\1_\AA$ is $\AA$-projective and $\AA$ has enough of injectives. Then 
\begin{equation}
H^*(\CC,M) \cong H^*(S^\bullet(\CC,M)).
\end{equation}
\end{theorem}
\begin{proof}
Similar to Theorem \ref{th:homology_of_augmented}. 
\end{proof}

Note that 
\begin{equation}
S^\bullet(\CC,\Triv(\1)) \cong \AA(\NN_\bullet(\CC),\1).
\end{equation}
Therefore, we obtain a corollary. 
\begin{corollary}\label{cor:cohomology}
Under the assumption of Theorem \ref{th:cohomology_of_augmented} we have an isomorphism
\begin{equation}
H^*(\CC) \cong H^*(\AA(\NN_\bullet(\CC),\1)).  
\end{equation}
\end{corollary}

\section{Magnitude homology}

\subsection{Magnitude homology of an enriched category} 
Let us recall the definition of the magnitude homology from \cite{leinster2021magnitude} (see also \cite{hepworth2022magnitude}). We will give a slightly more general definition, which is more natural for our setting. We don't assume that the base monoidal category $\VV$ is semicartesian. Instead of this we assume that the $\VV$-enriched category is augmented (Subsection \ref{subsection:nerve}). This is a generalisation because any category enriched over a semicartesian category has a unique augmentation.

Let $\VV$ be a symmetric monoidal category, and $\AA$ be an abelian cocomplete closed symmetric monoidal category. Assume that $\Sigma:\VV\to \AA$ is a strong monoidal functor. For an augmented $\VV$-category $\CC$ we consider a simplicial object $\widetilde\MC^\Sigma_\bullet(\CC)$ in $\AA,$ whose components are
\begin{equation}
\widetilde\MC^\Sigma_n(\CC) = \coprod_{c_0,\dots,c_n} \Sigma\CC(c_0,c_1)\otimes \dots \otimes \Sigma\CC(c_{n-1},c_n).
\end{equation}
The face maps $d_i:\widetilde\MC^\Sigma_n(\CC)\to \widetilde\MC^\Sigma_{n-1}(\CC)$ are defined for $0\leq i\leq n$ as follows. For $0<i<n$ they are induced by compositions
\begin{equation}
\Sigma\CC(c_{i-1},c_i)\otimes \Sigma\CC(c_i,c_{i+1}) \overset{\cong}\longrightarrow \Sigma(\CC(c_i,c_{i+1}) \otimes \CC(c_{i-1},c_i)) \overset{\Sigma (\circ)}\longrightarrow \Sigma\CC(c_{i-1},c_{i+1}).
\end{equation}
For $i=0,n$ they are induced by the augmentation 
\begin{equation}
 \Sigma\CC(c_0,c_1) \to \Sigma(\1)\cong \1, \hspace{1cm} \Sigma\CC(c_{n-1},c_n) \to \Sigma(\1)\cong \1.
\end{equation}
We denote by $\MC_\bullet^\Sigma(\CC)$ the normalised complex associated with the simplicial object $\widetilde\MC^\Sigma_\bullet(\CC).$ Then the magnitude homology is the homology of this complex
\begin{equation}
\MH^\Sigma_*(\CC)\cong H_*(\MC_\bullet^\Sigma(\CC)).
\end{equation}

If $\VV$ is a B\'enabou cosmos, and $\AA$ is an abelian B\'enabou cosmos, we can consider the $\AA$-category $\Sigma\CC$ defined in Subsection \ref{subsec:cosmos_change}. Using that $\Sigma$ is strong monoidal, we obtain that the morphism $\1_\AA\to \Sigma \1_\VV$ is an isomorphism. Therefore $\pt_\AA\cong  \Sigma \pt_\VV,$ and we obtain an augmentation $\Sigma\CC \longrightarrow \Sigma\pt_V \cong \pt_\AA.$

\begin{theorem}\label{th:magnitude_general}
Let $\VV$ be a B\'enabou cosmos, $\AA$ be an abelian B\'enabou cosmos, $\Sigma:\VV\to \AA$ be a strong monoidal functor and $\CC$ be an augmented $\VV$-category. Assume that $\AA$ has enough of projectives and $\Sigma\CC(c,c')$ is $\AA$-flat for any objects $c,c'$ of $\CC$. Then the category of $\Sigma \CC$-modules has enough of projectives and the magnitude homology is isomorphic to the homology of the augmented $\AA$-category $\Sigma\CC$
\begin{equation}
 \MH^\Sigma_*(\CC) \cong H_*(\Sigma\CC).
\end{equation}
\end{theorem}
\begin{proof}
It is easy to see that $\widetilde\MC^\Sigma_\bullet(\CC)$ is isomorphic to the nerve $\NN_\bullet(\Sigma \CC)$ of the augmented category $\Sigma\CC$. Since $\Sigma$ sends objects of $\VV$ to $\AA$-flat objects, the category $\Sigma\CC$ is $\AA$-flat. Then the statement follows from Corollary  \ref{cor:homology} and Proposition \ref{prop:projectives}. 
\end{proof}

Let us recall the definition of the magnitude cohomology \cite{hepworth2022magnitude}. It is defined in the same setting as homology. We consider a cosimplicial object of $\AA$ defined by the formula
\begin{equation}
\widetilde{\MC}^\bullet_\Sigma(\CC) = \AA(\widetilde{\MC}_\bullet^\Sigma(\CC),\1),
\end{equation}
and the associated normalised cochain complex
$
{\MC}^\bullet_\Sigma(\CC) = \AA({\MC}_\bullet^\Sigma(\CC),\1).$ The magnitude cohomology is defined as the cohomology of this cochain complex
\begin{equation}
    \MH^*_\Sigma(\CC)  = H^*(\MC^\bullet_\Sigma(\CC) ) \cong H^*(\widetilde{\MC}^\bullet_\Sigma(\CC))
\end{equation}

\begin{theorem}\label{th:magnitude_general2}
 Let $\VV$ be a B\'enabou cosmos, $\AA$ be an abelian B\'enabou cosmos, $\Sigma:\VV\to \AA$ be a strong monoidal functor and $\CC$ be an augmented $\VV$-category. Assume that $\AA$ has enough of injectives, $\Sigma\CC(c,c')$ is $\AA$-projective for any objects $c,c'$ of $\CC$ and $\1_\AA$ is $\AA$-projective. Then the category of $\Sigma \CC$-modules has enough of injectives and the magnitude cohomology is isomorphic to the cohomology of the augmented $\AA$-category $\Sigma\CC$
\begin{equation}
 \MH_\Sigma^*(\CC) \cong H^*(\Sigma\CC).
\end{equation}
\end{theorem}
\begin{proof}
Note that $\widetilde\MC^\Sigma_\bullet(\CC)\cong \NN_\bullet(\Sigma \CC).$  Since $\Sigma$ sends objects of $\VV$ to $\AA$-projective objects, the category $\Sigma\CC$ is $\AA$-projective. Then the statement follows from Corollary  \ref{cor:cohomology} and Proposition \ref{prop:injectives}.
\end{proof}

\begin{remark}[Magnitude homology with coefficients]
Note that under assumption of Theorem \ref{th:magnitude_general} we obtain a natural definition of the magnitude homology with coefficients in a $\Sigma\CC$-module $M$ 
\begin{equation}
\MH^\Sigma_*(\CC,M) := H_*(\Sigma\CC,M).
\end{equation}
Similarly, under the assumption of Theorem \ref{th:magnitude_general2} we obtain a natural definition of the magnitude cohomology with coefficients in a $\Sigma\CC$-module $M$ 
\begin{equation}
\MH_\Sigma^*(\CC,M) := H^*(\Sigma\CC,M).
\end{equation}
\end{remark}

\subsection{Distance modules over a quasimetric space} \label{quasi}

A generalized metric space is a set $X$ equipped with a distance map $d:X\times X\to [0,\infty]$ such that 
\begin{equation}
d(x,y)+d(y,z)\geq d(x,z), \hspace{1cm} d(x,x)=0
\end{equation}
for any $x,y,z\in X.$ A \emph{quasimetric space} is a generalized metric space such that
\begin{equation}
 d(x,y)=0 \  \Leftrightarrow \   x=y.
\end{equation}
The following theory could be developed for any generalized metric space, but for the case of a quasimetric space there are several simplifications. For example, the definition of invariants and coinvariants of distance modules for the general case of generalized metric spaces is more complicated.  So, for simplicity, further we will assume that $X$ is a quasimetric space. Elements of $X$ are called points. We say that a point $y$ lies  between points $x$ and $z,$ denoted by 
\begin{equation}
x\preccurlyeq  y\preccurlyeq z,    
\end{equation}
if $d(x,y)+d(y,z)=d(x,z).$ 

In this section we use the abelian B\'enabou cosmos $\RR\text{-}\GrMod(\KK)$  of $\RR$-graded modules over a commutative ring $\KK$ (for details see Subsection \ref{subsection:cosmos_of_graded_modules} in Appendix). A \emph{distance module} $M$ over a quasimetric space $X$ consists of: 
\begin{enumerate}
    \item a family $(M(x))_{x\in X}$ of $\RR$-graded modules indexed by points of $X;$
    \item  morphisms $M(x,y):M(x)\to M(y)$ of degree $d(x,y)$ for each pair of points $x,y$ such that $d(x,y)<\infty;$
\end{enumerate}
satisfying  conditions: $M(x,x)=1_{M(x)}$ and
\begin{equation}\label{eq:magnitude_definition}
M(y,z) \circ  M(x,y) = 
\begin{cases}
M(x,z), &  x\preccurlyeq  y\preccurlyeq z  \\
0, & \text{ else.}
\end{cases}
\end{equation}

An element  $m\in M(x)_\ell$ will be called an element of $M$ over the point $x$ of degree $\ell$. For such an element and any point $y$ such that $d(x,y)<\infty$ we set 
\begin{equation}
m\cdot y := M(x,y)(m).  
\end{equation}
So any element $m$ of $M$ over $x$ defines an element $m\cdot y$ over $y$ of degree 
\begin{equation}
|m\cdot y| = |m|+d(x,y).    
\end{equation}
Then the equation \eqref{eq:magnitude_definition} can be rewritten as
\begin{equation}
(m\cdot y)\cdot z 
=
\begin{cases}
m \cdot z, & x \preccurlyeq y \preccurlyeq z \\
0,& \text{ else.}
\end{cases}
\end{equation}

A morphism of distance modules $\varphi :M\to N$ is a family  $(\varphi (x):M(x)\to N(x))_{x\in X}$ of morphisms of $\RR$-graded modules of degree $0$ such that for any element $m$ of $M$ over a point $x$ and any point $y$ such that $d(x,y)<\infty$ we have 
\begin{equation}
\varphi(x)(m) \cdot y = \varphi(y)(m\cdot y).
\end{equation}
The category of distance modules over $X$ is denoted by $\DMod(X,\KK).$

For an $\RR$-graded module $A$ we denote by $\Triv(A)$ the distance module such that 
\begin{equation}\label{eq:triv:mag}
\Triv(A)(x)=A
\end{equation}
and $\Triv(A)(x,y)=0$ for any $x\neq y.$ This construction defines a functor 
\begin{equation}
\Triv : \RR\text{-}\GrMod(\KK) \longrightarrow \DMod(X,\KK).
\end{equation}

For a distance module $M$ we define the $\RR$-graded module of invariants as follows
\begin{equation}\label{eq:inv:mag}
\Inv(M)_\ell = \prod_{x\in X} \Inv_x(M)_\ell, 
\end{equation}
where
\begin{equation}
\Inv_x(M)_\ell = \{m\in M(x)_\ell \mid m\cdot y=0 \text{ for any } y\neq x, d(y,x)<\infty \}.
\end{equation}

Dually we define coinvariants
\begin{equation}\label{eq:coinv:mag}
\Coinv(M)_\ell = \bigoplus_{x\in X} \Coinv_x(M)_\ell, 
\end{equation}
where 
\begin{equation}\label{eq:coinv_x}
\Coinv_x(M)_\ell = M(x)_\ell/ \langle  m\cdot x\mid  m\in M(y)_{\ell-d(y,x)}, y\ne x, d(y,x)<\infty \rangle.
\end{equation}

These two constructions define functors
\begin{equation}
\Inv,\Coinv : \DMod(X,\KK)\longrightarrow \RR\text{-}\GrMod(\KK).
\end{equation}

Further we assume that  $\GGG\subseteq \RR$ is a subgroup such that $d(x,y)\in \GGG\cup\{\infty\}$ for any points $x,y\in X$ (for example, if $X$ is a digraph, we can take $\GGG=\ZZ$). Then we consider a full subcategory
\begin{equation}
\DMod_\GGG(X,\KK)\subseteq \DMod(X,\KK)
\end{equation}
that consists of distance modules such that $M(x)_\ell=0$ for any $\ell\in \RR\setminus \GGG.$ Such distance modules will be called $\GGG$-graded. Then the functors of invariants and coinvariants of distance modules, and the functor of trivial distance module can be restricted to the  functors
\begin{equation}
\Inv,\Coinv : \DMod_\GGG(X,\KK)\longrightarrow \GGG\text{-}\GrMod(\KK),
\end{equation}
\begin{equation}
\Triv  : \GGG\text{-}\GrMod(\KK) \longrightarrow \DMod_\GGG(X,\KK).
\end{equation}

Denote by $\Sigma_\GGG X$ an augmented category enriched over $\GGG\text{-}\GrMod(\KK)$ whose objects are points of $X$ and hom-objects are defined as
\begin{equation}\label{eq:sigmax1}
(\Sigma_\GGG X)(x,y)=\begin{cases}
\KK[d(x,y)], & d(x,y)<\infty \\
0, & d(x,y)=\infty.
\end{cases} 
\end{equation}
If we denote by $(x,y)$ the generator of $(\Sigma_\GGG X)(x,y),$ then $(x,x)=1_x$ and the composition is defined so that 
\begin{equation}\label{eq:sigmax2}
(y,z) \circ (x,y) = 
\begin{cases}
(x,z), & x \preccurlyeq y \preccurlyeq z \\
0, & \text{else}.
\end{cases}
\end{equation}
The augmentation 
$\varepsilon_{x,y} : (\Sigma_\GGG X)(x,y) \to \KK[0]$
is defined by the formula
\begin{equation}\label{eq:sigmax3}
\varepsilon_{x,y}((x,y)) =\begin{cases}
1, & x=y,\\
0,& x\ne y.
\end{cases}
\end{equation}

\begin{proposition}\label{prop:magmod-sigma}
The category of $\Sigma_\GGG X$-modules is isomorphic to the category of $\GGG$-graded distance modules 
\begin{equation}
\DMod_\GGG(X,\KK) \cong \Mod(\Sigma_\GGG X)
\end{equation}
and the functors of trivial distance module, invariants and coinvariants for distance modules defined by \eqref{eq:triv:mag}, \eqref{eq:inv:mag}  \eqref{eq:coinv:mag} correspond to the functors trivial module, invariants and coinvariants of modules defined by \eqref{eq:triv_inv_coinv}. In particular $\DMod_\GGG(X,\KK)$ has enough of projectives and injectives and there are adjunctions 
\begin{equation}
\Coinv \dashv \Triv \dashv \Inv.
\end{equation}
\end{proposition}
\begin{proof}
The isomorphism of categories is obvious. The fact that the functors of trivial modules coincide follows easily from the formula \eqref{eq:trivial}.

Let us prove that coinvariants in the sense of  distance modules defined by the formula \eqref{eq:coinv:mag} coincide with coinvariants in the sense of modules over enriched categories defined by the formula 
\eqref{eq:coinv}.  The formula  
\eqref{eq:coinv} implies that that we need to consider the cokernel of the difference of two morphisms  
\begin{equation}\tau-\tau': 
\bigoplus_{\substack{y,x\\ d(y,x)<\infty}} M(y)[d(y,x)]  \longrightarrow \bigoplus_x M(x) ,
\end{equation}
defined by the formulas $\tau\: {\rm em}_{y,x} ={\rm em}_y \:  \tilde \varepsilon_{y,x}$ and $\tau'\: {\rm em}_{y,x} = {\rm em}_x\: M(y,x),$ where
\begin{equation}
\tilde \varepsilon_{y,x} 
=
\begin{cases}
1_{M(y)}, & y=x\\
0, & y\ne x.
\end{cases}
\end{equation}
Note that both $\tau$ and $\tau'$ send the summand with the index $(y,x)$ to the summand with the index $x.$ Therefore we can present them as  $\tau= \bigoplus_x \tau_x $ and $\tau'= \bigoplus_x \tau'_x$ for  morphisms 
\begin{equation}
\tau_x,\tau'_x: \bigoplus_{y: d(y,x)<\infty} M(y)[d(y,x)] \to M(x).    
\end{equation}
Denote by $\tau_{x,y},\tau'_{x,y}:M(y)[d(y,x)]\to M(x)$ their restrictions to the corresponding summand. Then $\tau_{x,x}=1_{M(x)}=\tau'_{x,x},$  and if $x\neq y$ we have $\tau_{x,y}=0$ and $\tau'_{x,y}=M(y,x).$ Therefore
\begin{equation}
\Im(\tau_{x,y}-\tau'_{x,y}) = \begin{cases}
 0, & y=x\\
 \Im(M(y,x)), & y\neq x, d(y,x)<\infty.
\end{cases}
\end{equation}
It follows that 
\begin{equation}
\Coker(\tau_x - \tau'_x) = M(x) / \sum_{y\ne x,d(y,x)<\infty} \Im(M(y,x)).
\end{equation}
Therefore
\begin{equation}
\Coker(\tau_x - \tau'_x)_\ell = M(x)_\ell/\langle m\cdot x\mid  m\in M(y)_{\ell-d(y,x)}, y\ne x, d(y,x)<\infty\rangle.
\end{equation}
Hence these two definitions of coinvariants coincide. 

The proof for the case of invariants is similar.
\end{proof}

Further we consider the trivial magnitude $\tilde \KK=\Triv(\KK[0]).$

\begin{proposition}\label{prop:invariants_as_hom} For a distance module $M$ there is an isomorphism
\begin{equation}
\Hom(\tilde \KK[\ell],M) \cong \Inv(M)_\ell
\end{equation}
that sends $\varphi:\tilde\KK[\ell]\to M$ to $(\varphi(x)(1))_{x\in X}.$
\end{proposition}
\begin{proof}
Note that by the definition of the trivial distance module for any element $a$ of $\tilde \KK[\ell]$ over a point $x$ and $y\neq x$ we have $a\cdot y=0.$ Therefore for any  
$\varphi :\tilde\KK[\ell]\to M$ we have $\varphi(x)(a)\cdot y =0.$ 
Therefore $(\varphi(x)(1))_{x\in X}\in \Inv(M)_\ell.$ On the other hand, it is easy to see that for any $m\in \Inv(M)$ we can define a morphism $\varphi:\KK[\ell]\to M$ by the formula $\varphi(x)(a)=a\cdot m_x.$ The assertion follows. 
\end{proof}

\begin{remark}[Right vs. left distance modules]
The distance modules that we defined may be called right distance modules. Similarly one can define left distance modules so that we have morphisms $M(x,y):M(y)\to M(x)$ of degree $d(x,y),$ and the category of left distance modules over $X$ is equivalent to the category of $(\Sigma_\GGG X)^{op}$-modules. Note that $(\Sigma_\GGG X)^{op} = \Sigma_\GGG (X^{op}),$ where $X^{op}$ is a quasimetric space with the same set of points and a quasimetric defined by  $d^{X^{op}}(x,y)=d^X(y,x).$ Therefore, the category of left distance modules over $X$ is equivalent to the category of right distance modules over $X^{op}.$  
\end{remark}

\subsection{Magnitude homology of a quasimetric space} Let $X$ be a quasimetric space. For a tuple of points $(x_0,\dots,x_n)$ in $X$ we  set 
\begin{equation}
    |x_0,\dots,x_n| = \sum_{i=0}^{n-1} d(x_i,x_{i+1}).
\end{equation}

Let us recall the definition of the magnitude homology of a quasimetric space. For a real number $\ell\geq 0$ we consider a simplicial $\KK$-module $\widetilde{\MC}_{\bullet,\ell}(X)$ which is freely generated by tuples of points $(x_0,\dots,x_n)$ such that $|x_0,\dots,x_n|=\ell$
\begin{equation}
\widetilde{\MC}_{n,\ell}(X) = \KK\cdot \{(x_0,\dots,x_n)\mid |x_0,\dots,x_n|=\ell \}.   
\end{equation}
Tuples $(x_0,\dots,x_n)$ with $|x_0,\dots,x_n|<\ell$ are identified with zero in $\widetilde{\MC}_{n,\ell}(X).$ Then the $i$-th face map is the map of deletion of the $i$-th point, and the $i$-th degeneracy map is the map of doubling of $i$-th point. The assotiated normalized complex is denoted by $\MC_{\bullet,\ell}(X)$
\begin{equation}
\MC_{\bullet,\ell}(X) = \KK\cdot \{(x_0,\dots,x_n)\mid |x_0,\dots,x_n|=\ell, x_i\ne x_{i+1}\}.
\end{equation}
Then the magnitude homology of $X$ is defined as the homology of this chain complex
\begin{equation}
\MH_{n,\ell}(X) = H_n(\widetilde{\MC}_{\bullet,\ell}(X)) = H_n(\MC_{\bullet,\ell}(X)).
\end{equation}

Now we give a definition of magnitude homology with coefficients in a distance module $M.$ We consider a simplicial module 
$\widetilde{\MC}_{\bullet,\ell}(X,M)$  whose components are defined by the formula 
\begin{equation}
\widetilde{\MC}_{n,\ell}(X,M) = \bigoplus_{
\substack{|x_0,\dots,x_n|<\infty}
} M(x_0)_{\ell-|x_0,\dots,x_n|}.
\end{equation}
Here the summation is taken over all tuples $(x_0,\dots,x_n)$ of points of $X$ such that $|x_0,\dots,x_n|<\infty.$ For an element $m\in M(x_0)_{\ell-|x_0,\dots,x_n|}$ we denote by 
\begin{equation}
m \otimes (x_0,\dots,x_n)
\end{equation}
the corresponding element of $\widetilde{\MC}_{n,\ell}(X,M)$ from the summand indexed by $(x_0,\dots,x_n).$ 
The face maps $d_i$ for $0\leq i\leq n$ are defined as follows. For $1\leq i\leq n-1$ we have 
\begin{equation}\label{eq:partial_i}
d_{i}(m \otimes (x_0,\dots,x_n)) = 
\begin{cases}
m\otimes (x_0,\dots,\hat  x_i,\dots,x_n), &  x_{i-1} \preccurlyeq  x_i \preccurlyeq  x_{i+1} \\
0, & \text{ else,}
\end{cases}
\end{equation}
and for $i=0,n$ we have 
\begin{equation}\label{eq:partial_0}
d_0(m \otimes (x_0,\dots,x_n)) = 
m\cdot x_1\otimes (x_1,\dots,x_n).
\end{equation}
\begin{equation}
d_n(m \otimes (x_0,\dots,x_n)) = \begin{cases}
m \otimes (x_0,\dots,x_{n-1}), & x_n=x_{n+1} \\
0, & \text{else.}
\end{cases}
\end{equation}
The degeneracy maps $s_i$ are defined by
\begin{equation}
s_i(m\otimes (x_0,\dots,x_n)) = m\otimes (x_0,\dots,x_i,x_i,\dots,x_n). 
\end{equation}
The corresponding normalised complex  will be denoted by $\MC_{\bullet,\ell}(X)$
\begin{equation}
\MC_{n,\ell}(X,M) = \bigoplus_{
\substack{|x_0,\dots,x_n|<\infty \\ x_i\ne x_{i+1}}
} M(x_0)_{\ell-|x_0,\dots,x_n|}.
\end{equation}
The magnitude homology of $X$ with coefficients in $M$ is  defined by 
\begin{equation}
\MH_{n,\ell}(X,M) = H_n(\MC_{\bullet,\ell}(X,M)).
\end{equation}
It is easy to see that 
\begin{equation}
\MH_{n,\ell}(X)= \MH_{n,\ell}(X,\tilde \KK),
\end{equation}
where $\tilde \KK=\Triv(\KK[0]).$

\begin{theorem}\label{th:quasimetric} Let $\GGG\subseteq \RR$ be a subgroup such that 
$d(x,y)\in \GGG\cup\{\infty\}$ for any points $x,y$ of $X$ Then the category of $\GGG$-graded distance modules $\DMod_\GGG(X,\KK)$ has enough of projectives, the functor of coinvariants 
\begin{equation}
\Coinv : \DMod_\GGG(X,\KK)\longrightarrow \GGG\text{-}\GrMod(\KK)
\end{equation}
is right exact, and  the magnitude homology is isomorphic to its left derived functors  
\begin{equation}
\MH_{n,\ell}(X,M)\cong (L_n\Coinv)(M)_\ell.
\end{equation}
\end{theorem}
\begin{proof}
In order to prove this theorem we will use Theorem \ref{th:magnitude_general}. We consider the B\'enabou cosmos of non-negative real numbers with infinity $\VV=[0,\infty]^{op}$ and abelian B\'enabou cosmos $\AA=\GGG\text{-}\GrMod(\KK)$ of modules graded by real numbers. The strong monoidal functor $\Sigma_\GGG : [0,\infty]^{op} \to  \GGG\text{-}\GrMod(\KK)$ is defined so that $\Sigma_\GGG(\ell) = \KK[\ell]$ for $\ell<\infty$ and $\Sigma_\GGG(\infty)=0.$ It is a strong monoidal functor with the obvious isomorphism
$\Sigma_\GGG(\ell+\ell')\cong \Sigma_\GGG(\ell)\otimes \Sigma_\GGG(\ell').$ Also note that \eqref{eq:A[ell]} implies that the $\RR$-graded modules $\KK[\ell]$ are $\AA$-flat and $\AA$ projective. 
So the functor $\Sigma_\GGG$ sends objects of $\VV$ to $\AA$-flat and $\AA$-projective objects. The quasimetric space $X$ is treated as a $[0,\infty]$-category in the usual way. The augmentation $\varepsilon:X\to \pt$ is uniquely defined.  Then  $\Sigma_\GGG X$ is the  $\AA$-category defined by the formulas \eqref{eq:sigmax1}, \eqref{eq:sigmax2}, \eqref{eq:sigmax3}. So the
 statement follows from Proposition \ref{prop:magmod-sigma} and Theorem \ref{th:magnitude_general}. 
\end{proof}

\subsection{Magnitude cohomology of a quasimetric space} Let us recall the definition of the  magnitude cohomology of $X$ from \cite{hepworth2022magnitude}. For each real $\ell$ we consider the cosimplicial module
\begin{equation}
\widetilde{\MC}^{\bullet,\ell}(X) = \Hom(\widetilde{\MC}_{\bullet,\ell}(X),\KK)
\end{equation}
and the associated normalized cochain complex 
\begin{equation}
\MC^{\bullet,\ell}(X) = \Hom({\MC}_{\bullet,\ell}(X),\KK),
\end{equation}
whose cohomology is the magnitude cohomology of $X$
\begin{equation}
\MC^{n,\ell}(X)=H^n( \MC^{\bullet,\ell}(X)) \cong H^n( \widetilde{\MC}^{\bullet,\ell}(X) ). 
\end{equation}

Now we define the magnitude cohomology of $X$ with coefficients in a distance module $M$. For a real $\ell$ we consider a cosimplicial module  $\widetilde{\MC}^{\bullet,\ell}(\CC,M)$ such that 
\begin{equation}
\widetilde{\MC}^{n,\ell}(\CC,M) =  \prod_{|x_0,\dots,x_n|<\infty} 
M(x_n)_{-\ell+|x_0,\dots,x_n|}.
\end{equation}
For $\alpha\in \widetilde{\MC}^{n,\ell}(\CC,M)$ 
we denote by $\alpha(x_0,\dots,x_n)$ its projection in $M(x_n)_{-\ell+|x_0,\dots,x_n|}.$ Then the face maps $d^i:\widetilde{\MC}^{n,\ell}(\CC,M) \to \widetilde{\MC}^{n+1,\ell}(\CC,M)$ for $0\leq i\leq n+1$ are defined as follows. For $0<i<n+1$ we have
\begin{equation}
d^i(\alpha)(x_0,\dots,x_{n+1}) =
\begin{cases}
\alpha(x_0,\dots,\hat{x}_i,\dots,x_{n+1}),& x_{i-1} \preccurlyeq x_i \preccurlyeq x_{i+1}, \\
0, & \text{else,}
\end{cases}
\end{equation}
and for $i=0,n$ we have 
\begin{equation}
d^0(\alpha)(x_0,\dots,x_{n+1}) = \begin{cases}
\alpha(x_1,\dots,x_{n+1}), & x_0=x_1\\
  0,& \text{else},
\end{cases}
\end{equation}
and
\begin{equation}
d^{n+1}(\alpha)(x_0,\dots,x_{n+1}) = \alpha(x_0,\dots,x_n)\cdot x_{n+1}.
\end{equation}
The codegeneracy maps are defined by $s^i(\alpha)(x_0,\dots,x_{n-1}) = \alpha(x_0,\dots,x_i,x_i,\dots,x_n).$ The associated normalized cochain complex is denoted by $\MC^{\bullet,\ell}(X,M)$
\begin{equation}
{\MC}^{n,\ell}(X,M) =  \prod_{\substack{|x_0,\dots,x_n|<\infty \\  x_i\ne x_{i+1} }} 
M(x_n)_{-\ell+|x_0,\dots,x_n|}.
\end{equation}
The magnitude cohomology of $X$ with coefficients in $M$ is defined as cohomology of this complex
\begin{equation}
\MH^{n,\ell}(X,M)=H^n( \MC^{\bullet,\ell}(X,M))\cong H^n( \widetilde{\MC}^{\bullet,\ell}(X,M) ). 
\end{equation}
It is easy to see that 
\begin{equation}
\MH^{n,\ell}(X)\cong  \MH^{n,\ell}(X,\tilde \KK).
\end{equation}

\begin{theorem}\label{th:quasimetric2}
The category of $\GGG$-graded distance modules $\DMod_\GGG(X,\KK)$ has enough of injectives, the functor of invariants 
\begin{equation}
\Inv : \DMod_\GGG(X,\KK)\longrightarrow \GGG\text{-}\GrMod(\KK)
\end{equation}
is left exact, and  the magnitude cohomology is isomorphic to its right derived functors   
\begin{equation}\label{eq:mag_qm_cohom}
\MH^{n,\ell}(X,M)\cong (R^n\Inv)(M)_{-\ell}.
\end{equation}
\end{theorem}
\begin{proof}
The proof is similar to the proof of Theorem \ref{th:quasimetric}.
\end{proof}

\begin{remark}
Note that in the formula  \eqref{eq:mag_qm_cohom} we have $-\ell$ in the subscript of the right-hand term. For this reason, we write $\ell$ in superscript in the notation for magnitude cohomology  $\MC^{n,\ell}(X,M)$.
\end{remark}

\begin{corollary}\label{cor:cohom-ext}
For any $\GGG$-graded distance module $M$ there magnitude cohomology with coefficients in $M$ can be presented as the Ext in the category of distance modules 
\begin{equation}
\MH^{n,\ell}(X,M)\cong \Ext^{n}(\tilde \KK,M[\ell]).
\end{equation}
\end{corollary}
\begin{proof}
It follows from Theorem \ref{th:quasimetric2} and Proposition \ref{prop:invariants_as_hom}. 
\end{proof}

\section{Magnitude homology of a finite quasimetric space}

\subsection{Distance algebra of a finite quasimetric space} 

Let $X$ be a finite quasimetric space (\S \ref{quasi}) and $\GGG\subseteq \RR$ be a subgroup of the additive group of real numbers such that $d(x,y)\in \GGG\cup\{\infty\}$ for any points $x,y$ of $X.$ Consider a $\GGG$-graded algebra $\sigma X$ which as a $\KK$-module is freely generated by pairs of points $(x,y)$ such that $d(x,y)<\infty$ and the product is defined by 
\begin{equation}
(x,y)\cdot (y',z) = \begin{cases}
(x,z), & y=y', x \preccurlyeq y \preccurlyeq z\\
0, & \text{else.}
\end{cases}
\end{equation}
The algebra $\sigma X$ will be referred to as the distance algebra. Note that the elements $e_x:=(x,x)$ are orthogonal idempotents and $1=\sum_{x\in X} e_x$ is the unit of this algebra. So the set $\{e_x\mid x\in X\}$ is a complete set of orthogonal idempotents. Note that in order to define the unit, we need to use that $X$ is finite. The $\GGG$-grading is defined by the formula $|(x,y)|=d(x,y).$ In other words, $(\sigma X)_\ell$ is spanned by pairs $(x,y)$ such that $d(x,y)=\ell.$ In particular, $(\sigma X)_0$ is spanned by the idempotents $e_x.$ 

\begin{proposition}\label{prop:magmod_equivalence}
There is an equivalence of categories 
\begin{equation}
\DMod_\GGG(X) \simeq \GGG\text{-}\GrMod(\sigma X), \hspace{1cm} M\mapsto \bar M
\end{equation}
that sends a $\GGG$-graded distance module $M$ to a $\GGG$-graded right $\sigma X$-module 
\begin{equation}
\bar M=\bigoplus_{x\in X} M(x),
\end{equation}
where the structure of $\sigma X$-module is defined by $m\cdot (x,y)=m\cdot y,$ if $m\in M(x),$ and by $m\cdot (x,y)=0,$ if $m\in M(x'), x\neq x'.$   
\end{proposition}
\begin{proof}
It is easy to check that the construction $M\mapsto \hat M$ from the statement  defines a functor $\DMod_\GGG(X) \to \GGG\text{-}\GrMod(\sigma X)$. Let us construct the functor in the opposite direction. If $N$ is a graded right $\sigma X$-module, we construct a distance module $\hat N$ so that $\hat N(x)=Ne_x$ and $n\cdot y=n\cdot (x,y)$ for $n\in \hat{N}(x).$ The fact that  $\{e_x\mid x\in X\}$ is a complete set of orthogonal idempotents implies that  $N=\bigoplus_{x\in X} Ne_x.$ Using this, the proof that $N\mapsto \hat N$ is a functor opposite to the functor $M\mapsto \bar M$ is straightforward. 
\end{proof}

\subsection{Magnitude homology of a finite quasimetric space}
For a finite quasimetric space $X$ we denote by $JX$ the homogeneous ideal of $\sigma X$ generated by all pairs of positive degree. It is easy to check that if $n$ is the cardinality of $X,$ then 
\begin{equation}
(JX)^n=0.
\end{equation}
Therefore, $JX$ is a nilpotent ideal. We set 
\begin{equation}
S = \sigma X / JX.
\end{equation}
As an algebra $S$ is isomorphic to the product $\prod_{x\in X}\KK,$ where each factor is generated by the corresponding idempotent $e_x.$ If $\KK$ is a principal ideal domain, then the Jacobson radical of $S$ is trivial. Using that $JX$ is nilpotent, in this case we obtain that $JX$ is the Jacobson radical of $\sigma X$ (that is why this notation was chosen).

Since $S$ is a quotient graded algebra of $\sigma X$, it can be treated as a graded right $\sigma X$-module and a graded left $\sigma X$-module. This graded module is concentrated in degree zero and it is equal to the direct sum of graded submodules 
\begin{equation}
S = \bigoplus_{x\in X} S_x,
\end{equation}
where $S_x=\KK \cdot e_x.$ Using this formula we obtain that 
$S$ corresponds to the trivial distance module $\tilde \KK=\Triv(\KK[0])$ 
\begin{equation}\label{eq:S-bar}
S=\bar{\tilde \KK}.    
\end{equation}

\begin{proposition}\label{prop:coinvariants_finite}
If $M$ is a $\GGG$-graded distance module over $X$ and $\bar M$ is the corresponding graded $\sigma X$-module, then 
\begin{equation}
\Coinv(M)\cong \bar M\otimes_{\sigma X} S, \hspace{1cm} \Inv(M)\cong \Hom_{\sigma X}(S,\bar M)
\end{equation}
\end{proposition}
\begin{proof}
Since $S=\sigma X/JX,$ we have $\bar M\otimes_{\sigma X}S\cong \bar M/(\bar M \cdot JX).$ Since $\bar M=\bigoplus_x M(x),$ we obtain that $\bar M/(\bar M\tilde JX)$ is a direct sum of $\KK$-modules 
\begin{equation}
 M(x)/ \left(\sum_{y\neq x, d(y,x)<\infty}  M(y)\cdot (y,x)\right).
\end{equation}
Comparing this with the formula \eqref{eq:coinv_x}, we obtain the isomorphism $\Coinv(M)\cong \bar M\otimes_{\sigma G} S.$ The second isomorphism can be proved similarly.
\end{proof}

In the statement of the next theorem we use bigraded Tor and Ext functors. For the definition see Subsection \ref{subsection:graded_modules_over_graded_algebras} of the Appendix.

\begin{theorem}\label{th:magnitude_homology_finite} Let $\KK$ be a commutative ring, $\GGG\subseteq \RR$ be a subgroup, $X$ be a finite quasimetric space such that $d(x,y)\in \GGG \cup \{\infty\},$ and $M$ be a $\GGG$-graded distance module over $X$. Then for any $(n,\ell)\in \ZZ\times \GGG$ there are isomorphisms 
\begin{equation}
\MH_{n,\ell}(X,M) \cong {\rm Tor}^{\sigma X}_{n,\ell}(\bar M,S), \hspace{1cm} \MH^{n,\ell}(X,M)\cong {\rm Ext}^{n,\ell}_{\sigma X}(S,\bar M).
\end{equation}
In particular, we obtain
\begin{equation}
\MH_{n,\ell}(X) \cong {\rm Tor}^{\sigma X}_{n,\ell}(S,S), \hspace{1cm} \MH^{n,\ell}(X)\cong {\rm Ext}^{n,\ell}_{\sigma X}(S,S).
\end{equation}
\end{theorem}
\begin{proof}
It follows from Theorems \ref{th:quasimetric},  \ref{th:quasimetric2},  Propositions \ref{prop:magmod_equivalence},  \ref{prop:coinvariants_finite}, and equation \eqref{eq:S-bar}.
\end{proof}

\begin{remark}[Bar resolution of $S$]\label{remark:projective}
The statement of Theorem  \ref{th:magnitude_homology_finite} can be extended by constructing an explicit projective resolution of the graded right $\sigma X$-module $S$. The resolution from Proposition \ref{prop:bar_resolution} together with the equivalence from Proposition \ref{prop:magmod_equivalence} defines a graded projective resolution $P_\bullet$ of the graded $\sigma X$-module $S$ such that 
\begin{equation}
(P_n)_\ell =\KK\cdot  
\{(x_0,\dots,x_n,y)\mid |x_0,\dots,x_n,y|=\ell \},
\end{equation}
and the differential is defined by $\partial=\sum_{i=0}^n (-1)d_i,$ where $d_i$ is defined by 
\begin{equation}
d_i(x_0,\dots,x_n,y)= 
\begin{cases}
(x_0,\dots,\hat x_i,\dots,x_n,y), & |x_0,\dots,\hat x_i,\dots,x_n,y|=|x_0,\dots,x_n,y|\\
0, & \text{else.}
\end{cases}
\end{equation}
The structure  of $\sigma X$-module on $P_n$ is defined by $(x_0,\dots,x_n,y)\cdot (y',z)=(x_0,\dots,x_n,z),$ if $|x_0,\dots,x_n,y|=|x_0,\dots,x_n,z|$ and $y=y',$ and by $(x_0,\dots,x_n,y)\cdot (y',z)=0$ otherwise. The module $P_n$ is projective, because there is an isomorphism of $\sigma X$-modules
\begin{equation}
P_n\cong \bigoplus_{|x_0,\dots,x_n|<\infty} (e_{x_n}\cdot \sigma X)[|x_0,\dots,x_n|],   
\end{equation}
where $e_{x_n}\cdot \sigma X$ is the projective  summand of $\sigma X$ corresponding to the idempotent $e_{x_n}.$ 
\end{remark}

\begin{remark}
Apparently, the isomorphism $\MH_{n,\ell}(X,M)\cong \Tor^{\sigma X}_{n,\ell}(\bar M,S)$ could be generalized to the case of any quasimetric space. In order to do this, for a right distance module $M$ over $X$ and a left distance module $N$ over $X$ we need to introduce the tensor product $M\otimes_X N,$ which is an $\RR$-graded module. Then we need to define the Tor functor of distance modules $\Tor^X_n$ as the derived functor and show that $\MH_{n,\ell}(X,M)\cong \Tor^X_{n}(M,\tilde \KK)_\ell.$
\end{remark}

\subsection{Magnitude cohomology as a Yoneda algebra}  Usually the Yoneda product is defined in terms of $n$-fold extensions \cite[Ch.III,\S 5]{maclane2012homology} \cite{buchsbaum1959note}, \cite{yoneda1958note}, \cite{yoneda1954homology} but we use the description in terms of projective resolutions that can be found in \cite[\S 6 Exercise 2]{maclane2012homology}.

Let  $\AA$ be an abelian category with enough of projectives and $a,b,c$ are its objects. Then the Yoneda product 
\begin{equation}
\Ext^m(b,c) \times \Ext^n(a,b) \longrightarrow \Ext^{n+m}(a,c)
\end{equation}
is defined as follows. Let $p_\bullet$ be a projective resolution of $a$ and $q_\bullet$ be a projective resolution of $b.$ Assume that $\alpha:p_n\to b$ and $\beta:q_m\to c$ are two cycles representing elements from $\Ext^n(a,b)=H^n(\Hom(p_\bullet,b))$ and $\Ext^m(b,c)=H^m(\Hom(q_\bullet,c)).$ The fact that $\alpha$ is a cocycle means that $\alpha \partial^p_{n+1}=0.$ Therefore, $\alpha$ induces a morphism $\alpha': \Coker(\partial^p_{n+1})\longrightarrow b.$ This morphism can be lifted to the morphism of resolutions uniquely up to homotopy.
\begin{equation}
\begin{tikzcd}
\dots \ar[r]  & p_{n+m}\ar[r] \ar[d,"\alpha_{m}"] & p_{n+m-1}\ar[r] \ar[d,"\alpha_{m-1}"] & \dots \ar[r] & p_n \ar[d,"\alpha_0"] \ar[r] & \Coker(\partial^p_{n+1})\ar[d,"\alpha'"] \\
 \dots \ar[r] & q_m \ar[r] & q_{m-1}\ar[r] & \dots \ar[r] & q_0\ar[r] & b 
\end{tikzcd}
\end{equation}
Then the Yoneda product is defined as  
\begin{equation}
[\beta]\circ [\alpha] = [\beta\circ \alpha_m].
\end{equation}
The Yoneda product is bilinear. This allows us to define a structure of a graded algebra on ${\rm Ext}^*_\AA(a,a),$ which is called Yoneda algebra of $a$
\begin{equation}
\mathcal{Y}^*_\AA(a) = \Ext^*_\AA(a,a).
\end{equation}

If $\GGG$ is an abelian group,  $\Lambda$ is a $\GGG$-graded algebra and $M$ is a $\GGG$-graded $\Lambda$-module, then the $\ZZ\times\GGG$-graded Yoneda algebra of $M$ is defined as
\begin{equation}
\mathcal{Y}^{*,*}_\Lambda(M)= \Ext^{*,*}_\Lambda(M,M) = \bigoplus_{\ell\in \GGG} \Ext^*_{\GGG\text{-}\GrMod(\Lambda)}(M,M[\ell]),
\end{equation}
where the product is also defined by the Yoneda product (see \eqref{eq:ext_bigrad_as_ext_in_cat}). If $P_\bullet$ is a projective resolution of $M,$ then $P_\bullet[\ell]$ is a projective resolution of $M[\ell].$ An element of 
 $\Ext^{n,\ell}_\Lambda(M,M)$ can be presented by a homomorphism $\alpha:P_n\to M[\ell]$ such that $\alpha \partial^P_{n+1} =0.$ Then for another such a homomorphism   $\beta:P_n \to M[s]$ we have  
\begin{equation}
[\beta] \circ [\alpha] = [\beta[\ell] \circ \alpha_m].
\end{equation}
All these definitions work similarly for left and right modules. 

Now let us recall the definition of the product for the magnitude cohomology of a quasimetric space $X$ from  \cite{hepworth2022magnitude}. First we define the structure of a dg-algebra on
\begin{equation}
\MC^{\bullet,*}(X) = \bigoplus_{\ell\in \GGG} \MC^{\bullet,\ell}(X).    
\end{equation}
If $\varphi\in \MC^{n,\ell}(X)$ and $\psi\in \MC^{m,s}(X),$ then $\varphi\cdot \psi \in \MC^{n+m,\ell+s}(X)$ is defined by 
\begin{equation}
( \psi\cdot  \varphi)(x_0,\dots,x_{n+m}) = 
\psi(x_0,\dots,x_m)\cdot \varphi(x_m,\dots,x_{n+m})    
\end{equation}
if  $|x_0,\dots,x_n|=\ell$ and  $|x_n,\dots,x_{n+m}|=s;$ and by 
\begin{equation}
(\varphi \cdot \psi )(x_0,\dots,x_{n+m})=0 
\end{equation}
otherwise. This product induces a product on the magnitude cohomology that defines a structure of a $\ZZ\times\GGG$-graded algebra on $\MH^{*,*}(X).$

\begin{theorem}\label{th:yoneda} Let $\KK$ be a commutative ring, $\GGG\subseteq \RR$ be a subgroup and $X$ be a finite quasimetric space such that $d(x,y)\in \GGG \cup \{\infty\}.$
Then the $\ZZ\times\GGG$-graded magnitude cohomology algebra of $X$ is isomorphic to the $\ZZ\times\GGG$-graded Yoneda algebra of the $\GGG$-graded left $\sigma X$-module $S$
\begin{equation}
\MH^{*,*}(X) \cong  \mathcal{Y}^{*,*}_{\sigma X}(S).
\end{equation}
\end{theorem}
\begin{remark} In Theorem \ref{th:yoneda} we treat $S$ as a left module. If we treated $S$ as a right module, a similar proof would give an isomorphism $\MH^{*,*}(X)\cong \mathcal{Y}^{*,*}_{\sigma X}(S)^{op},$ where ``op'' means that we consider the opposite product defined by $ab:= (-1)^{nm} b\circ a$
for $a\in \mathcal{Y}^{n,\ell}(S)$ and $b\in \mathcal{Y}^{m,s}(S).$ In order to avoid this construction of opposite algebra in the statement of Theorem \ref{th:yoneda} and to simplify the proof, we decided to formulate it in terms of left modules.
\end{remark}
\begin{proof}[Proof of Theorem \ref{th:yoneda}] Set $A=\sigma  X.$ In this proof we deal with left $\GGG$-graded $A$-modules. We need a left-module version of the projective resolution $P_\bullet$ of $S$ from Remark \ref{remark:projective}.  Its components are defined by 
\begin{equation}
(P_n)_\ell =\KK\cdot  
\{(y,x_0,\dots,x_n)\mid |y,x_0,\dots,x_n|=\ell \},
\end{equation}
and the differential is defined by $\partial=\sum_{i=0}^n (-1)d_i,$ where $d_i$ is given by 
\begin{equation}
d_i(y,x_0,\dots,x_{n})= 
\begin{cases}
(y,x_0,\dots,\hat x_i,\dots,x_{n}), & |y,x_0,\dots,\hat x_i,\dots,x_{n}|=|y,x_0,\dots,x_{n}|\\
0, & \text{else.}
\end{cases}
\end{equation}
The structure of left $A$-module on $P_n$ is defined by $(z,y')(y,x_0,\dots,x_n)=(z,x_0,\dots,x_n),$ if $|z,x_0,\dots,x_n|=|y,x_0,\dots,x_n|$ and $y=y',$ and by $(z,y')\cdot (y,x_0,\dots,x_n)=0$ otherwise. There is an isomorphism of $A$-modules
\begin{equation}\label{eq:isom_p_n_left}
P_n\cong \bigoplus_{|x_0,\dots,x_n|<\infty} Ae_{x_0}[|x_0,\dots,x_n|].  
\end{equation}
Note that 
\begin{equation}
\Hom_A(Ae_x,S_y) \cong \begin{cases}
\KK[0], & x=y\\
0, & x\neq y.
\end{cases}
\end{equation}
Here $\Hom_A(Ae_x,S_x)_0$ is generated by the projection $Ae_x\epi S_x.$
It follows that $\Hom_A(Ae_x,S)\cong \KK[0]$ for any $x.$ Using the isomorphism \eqref{eq:isom_p_n_left} we obtain 
\begin{equation}
\widetilde{\MC}^{n,\ell}(X)\cong \Hom_A(P_n,S)_{-\ell},
\end{equation}
where a map $\varphi:\MC_{n,\ell}(X) \to \KK$ corresponds to a map $\hat \varphi:P_n \to S[\ell]$ defined by 
\begin{equation}
\hat \varphi(y,x_0,\dots,x_n)= e_{x_0} \cdot \varphi(x_0,\dots,x_n).
\end{equation}
Here we assume that 
\begin{equation}
\varphi(x_0,\dots,x_n)=0 \hspace{1cm} \text{if} \hspace{1cm} |x_0,\dots,x_n|\neq \ell.
\end{equation}
It is easy to check that this isomorphism is compatible with the differentials, and we have an isomorphism of complexes
\begin{equation}
\widetilde{\MC}^{\bullet,\ell}(X)\cong \Hom_A(P_\bullet,S)_{-\ell},
\end{equation}
which induces an isomorphism for cohomology
\begin{equation}
\MH^{n,\ell}(X)\cong \Ext^{n,\ell}_A(S,S).
\end{equation}
Let us prove that this isomorphism is an isomorphism of bigraded algebras.

Assume that we have an element in the magnitude cohomology $\MH^{n,\ell}(X).$ Take a representing cocycle  $\varphi:\widetilde{\MC}_{n,\ell}(X)\to \KK.$ Then $\hat \varphi:P_n\to S[\ell]$ is also a cocycle and it induces a homomorphism $\hat \varphi':\Coker(\partial^P_{n+1}) \to S[\ell].$ Consider a homomorphism $\hat \varphi_0:P_{n}\to P_0[\ell]$ defined by the formula 
\begin{equation}
\hat \varphi_0(y,x_0,\dots,x_{n}) =
  (y,x_0) \cdot \varphi (x_0,\dots,x_n).
\end{equation}
It is easy to check that $\hat \varphi_0$ is a homomorphism of right $A$-modules such that the diagram 
\begin{equation}
\begin{tikzcd}
P_n \ar[r,twoheadrightarrow] \ar[d,"\hat\varphi_0"'] \ar[rd,"\hat \varphi"]
& 
\Coker(\partial^P_{n+1}) \ar[d,"\hat \varphi'"] 
\\ 
P_0[\ell] \ar[r,twoheadrightarrow] 
& 
S[\ell]
\end{tikzcd}
\end{equation}
is commutative. Further for any $k\geq 0$ we define 
$\hat \varphi_k:P_{k+n}\to P_k $
by the formula 
\begin{equation}
\hat \varphi_k(y,x_0,\dots,x_{n+k}) =
(y,x_0,\dots,x_k)\cdot  \varphi (x_k,\dots,x_{n+k}).
\end{equation} 
We claim that $\hat \varphi_k$ defines a morphism of chain complexes: 
\begin{equation} \label{eq:commute_differential}
\hat \varphi_{k-1} \partial_{n+k}=\partial_k \hat \varphi_k, \hspace{1cm} k\geq 1.
\end{equation}
If we prove this, the statement of the theorem follows easily, because if $\psi:\widetilde{\MC}_{m,s}(X)\to \KK$ is another cocycle, then it is easy to check that 
\begin{equation}
 \hat \psi[\ell] \circ \hat \varphi_m = \widehat{\psi\cdot \varphi}.
\end{equation}
So the rest of the proof is devoted to the proof of the equation 
\eqref{eq:commute_differential}. 
 
If $n=0,$ then the statement is obvious, so we will assume  $n\geq 1.$ The differentials are defined by the formulas
\begin{equation}\label{eq:formulas_for_differentials}
\partial_{k+n} = \sum_{i=0}^{n+k} (-1)^i d_i, \hspace{1cm} \partial_k = \sum_{i=0}^k (-1)^i d_i.
\end{equation}
First we note that 
\begin{equation}\label{eq:commute1}
\hat{\varphi}_{k-1} d_i = d_i \hat{\varphi}_k, \hspace{1cm} 0\leq i \leq k-1. 
\end{equation}
Indeed, both maps applied to $(y,x_0,\dots,x_{k+n})$ give $(y,x_0,\dots,\hat{x_i},\dots,x_k)\cdot \varphi(x_k,\dots,x_{k+n}),$ if $x_{i-1} \preccurlyeq x_{i} \preccurlyeq x_{i+1},$ and $0$ otherwise (here we assume  $x_{-1}=y$). On the other hand,  we have 
\begin{equation}
\begin{split}
\hat \varphi_{k-1} d_{i}(y,x_0,\dots,x_{k+n}) =&\\ (y,x_0,\dots,x_{k-1})\cdot \varphi(d_{i-k+1}(x_{k-1},\dots,x_{k+n})),& \hspace{1cm}  k \leq i\leq k+n. 
\end{split}
\end{equation}
Then, using the fact that $\varphi$ is a cycle, we obtain
\begin{equation}\label{eq:commute2}
\begin{split}
\hat{\varphi}_{k-1} (\sum_{i=k}^{k+n}(-1)^i d_i(y,x_0,\dots,x_{n+k})) = \\
(-1)^{k} (y,x_0,\dots,x_{k-1})\cdot \varphi( d_0(x_{k-1},\dots,x_{k+n}) ).&
\end{split}
\end{equation}
Finally we note that 
\begin{equation}\label{eq:commute3}
 d_k\hat{\varphi}_{k}(y,x_0,\dots,x_{n+k}) = (y,x_0,\dots,x_{k-1})\cdot \varphi( d_0(x_{k-1},\dots,x_{k+n})).
\end{equation}
Indeed, if $x_{k-1}\neq x_k,$ then both parts of the equation are zero, because $d_0(x_{k-1},\dots,x_{k+n})=0$ and $d_k(y,x_0,\dots,x_k)=0;$ and if $x_{k-1}=x_k$ then both parts of the equation are equal to $(y,x_0,\dots,x_{k-1}) \cdot \varphi(x_k,\dots,x_{n+k}).$ The equation \eqref{eq:commute_differential} follows from the equations \eqref{eq:commute1},\eqref{eq:commute2},\eqref{eq:commute3}.
\end{proof}

\begin{remark}
Similarly to Theorem \ref{th:yoneda} one can prove that the product in the magnitude cohomology algebra of any quasimetric space $X$ can be presented via Yoneda product, if we present the magnitude cohomology in terms of Ext functors in the category of distance modules  
$\MC^{*,*}(X)\cong \bigoplus_{\ell\in \GGG} \Ext^*(\tilde \KK,\tilde \KK[\ell])$  (Corollary \ref{cor:cohom-ext}).
\end{remark}

\section{Magnitude homology of a digraph}

\subsection{Background on representations of quivers}

A quiver is a quadruple $Q=(Q_0,Q_1,s,t)$ such that $Q_0,Q_1$ are sets and $s,t:Q_1\to Q_0$ are maps. Elements of $Q_0$ are called vertices, elements of $Q_1$ are arrows, for an arrow  $\alpha\in Q_1$ we say that $s(\alpha)$ is the start of $\alpha,$ and $t(\alpha)$ is the target of $\alpha.$ A quiver $Q$ is called finite, if $Q_0,Q_1$ are finite. 

Let $\KK$ be a commutative ring.  A representation $V$ of a quiver $Q$ is a collection of $\KK$-modules indexed by vertices $V(x)$ and a collection of homomorphisms $V(\alpha):V(x)\to V(y)$ for each arrow $\alpha,$ where $x=s(\alpha)$ and $y=t(\alpha).$  A morphism of representations $f:V\to U$ is a collection of homomorphisms $f(x):V(x)\to U(x)$ indexed by vertices such that $U(\alpha)f(x) = f(y) V(\alpha)$ for any arrow $\alpha$ from $x$ to $y.$ The category of representations will be denoted by $\Rep(Q).$

A path of length $n\geq 1$ in a quiver $Q$  is a sequence of arrows $\alpha_1\dots \alpha_n$ such that $t(\alpha_i)=s(\alpha_{i+1}).$ With each vertex $x\in Q_0$ we also associate a trivial path $e_x$ of length zero. For a finite quiver $Q$ we define the path algebra $\KK Q$ as a $\KK$-algebra, whose $\KK$-basis consists of all paths and trivial paths. The product is defined by concatenation of paths, and it vanishes if the concatenation does not exist. The algebra has a unit defined as the sum of trivial paths  $1=\sum_{x\in Q_0} e_x.$ Note that in order to define the unit we use that $Q_0$ is finite. 
It is well-known that the category of representations of a finite quiver $Q$ is equivalent to the category of right $\KK Q$-modules  
\begin{equation}\label{eq:rep-equiv}
\Rep(Q)\simeq \Mod(\KK Q)
\end{equation}
(see  \cite[Ch.III, Th.1.5]{auslender}, \cite[\S 1, Lemma]{crawley1992lectures},  \cite[Theorem 2.4.1]{hazewinkel2006algebras2}).
This equivalence sends a representation $V$ to a $\KK Q$-module $M=\bigoplus_x V(x)$ such that 
$v \cdot \alpha = V(\alpha)(v)$ for any $v\in V(x)$ and any arrow $\alpha$ such that $s(\alpha)=x;$ and $v\cdot \alpha=0,$ if $s(\alpha)\neq x.$ The functor in the opposite direction sends a module $M$ to a representation $V$ defined by $V(x) = Me_x$ and $V(\alpha)(m)=m\cdot \alpha.$

A \emph{quiver with relations} is a pair $(Q,I),$ where $Q$ is a finite quiver and $I$ is an ideal of $\KK Q$. We say that a representation $V$ of a quiver satisfies relations from $I,$ if the annihilator of the corresponding $\KK Q$-module contains $I.$ The full subcategory of representations satisfying relations from $I$ is denoted by 
\begin{equation}
\Rep(Q,I) \subseteq \Rep(Q).
\end{equation}
The equivalence \eqref{eq:rep-equiv} can be restricted to an equivalence 
\begin{equation}
\Rep(Q,I) \simeq \Mod( \KK Q/I)
\end{equation}
(see \cite[Ch.III, Prop.1.7]{auslender}, \cite[Ch.III, Th.1.6]{assem2006elements}).

All this theory can be developed in the graded setting. In this section we assume that $\GGG=\ZZ$ and ``graded'' means ``$\ZZ$-graded''. A graded representation $V$ of a quiver $Q$ is a collection of graded $\KK$-modules indexed by vertices $V(x)$ and a collection of homomorphisms $V(\alpha):V(x)\to V(y)$ of degree $1$ for each arrow $\alpha,$ where $x=s(\alpha)$ and $y=t(\alpha).$  A morphism of graded representations $f:V\to U$ is a collection of homomorphisms $f(x):V(x)\to U(x)$ of degree zero indexed by vertices such that $U(\alpha)f(x) = f(y) V(\alpha)$ for any arrow $\alpha$ from $x$ to $y.$ The category of graded representations will be denoted by $\GrRep(Q).$ 
The path algebra $\KK Q$ has a natural grading, where the $n$-th homogeneous component $(\KK Q)_n$ consists of linear combinations of paths of length $n.$ Similarly to the non-graded case we obtain an equivalence
\begin{equation}
\GrRep(Q)\simeq \GrMod(\KK Q).
\end{equation}
If we assume that $I$ is a homogeneous ideal of $\KK Q$, we can also similarly define a subcategory 
\begin{equation}
\GrRep(Q,I) \subseteq \GrRep(Q)
\end{equation}
such that there is an equivalence
\begin{equation}\label{eq:graded_rep}
\GrRep(Q,I) \simeq \GrMod(\KK Q/I).
\end{equation}

Set $\Lambda=\KK Q/I.$ Denote by $J$ the ideal of $\KK Q$ generated by all non-trivial paths. The ideal $I$ is called admissible if 
\begin{equation}\label{eq:J^nIJ^2}
J^n \subseteq I \subseteq J^2
\end{equation}
for some $n$ (see \cite[Ch. II, Def. 2.1]{assem2006elements}). In this case $(Q,I)$ is called bound quiver and $\Lambda$ is the bound quiver algebra (usually this terminology is used assuming $\KK$ is a field). In this case $\Lambda$ is finitely generated as a $\KK$-module. 
The quotient 
\begin{equation}
    S = \KK Q/J
\end{equation}
has a natural structure of a (right and left) $\Lambda$-module. Note that $S$ is the direct sum of submodules $S_x\cong \KK e_x$
\begin{equation}\label{eq:simple_decomp}
    S = \bigoplus_{x\in Q_0} S_x.
\end{equation}
If $\KK$ is a field and $I$ is admissible, the bound quiver algebra $\Lambda$ is finite dimensional and  $J(\Lambda)=J/I$ is the Jacobson radical of $\Lambda$ (see  \cite[\S III.1, Prop. 1.6]{auslender}, \cite[Ch.II, Lemma 2.1]{assem2006elements}).  Moreover, in this case $\{S_x\mid x\in Q_0\}$ is a complete set of representatives of the isomorphism classes of the simple $\Lambda$-modules  and $S$ is the direct sum this set \cite[Ch. III, Lemma 2.1]{assem2006elements}.

\subsection{Distance modules as representations of quivers}

We define a digraph $G$ as a quiver without loops and multiple arrows. In this case arrows can be identified with pairs of vertices $\alpha = (s(\alpha),t(\alpha)).$ So we can assume that $G_1\subseteq G_0\times G_0$ and $s,t$ are projections. A finite digraph is a digraph such that $G_0$ is finite. Further in this section we assume that $G$ is a fixed finite digraph.

A path $\alpha_1\dots \alpha_n$ in the path algebra $\KK G$ can be identified with a sequence of the vertices  $(x_0,\dots,x_n).$ The trivial path $e_x$ is identified with $(x).$ A path $(x_0,\dots,x_n)$ is called shortest, if it is a shortest path from $x_0$ to $x_n.$ A path which is not shortest  will be called non-shortest. 
Denote by $R(G)\triangleleft \KK G$ a homogeneous ideal generated by two types of relations
\begin{itemize}
\item[(R1)] $(x_0,\dots,x_n)=(y_0,\dots,y_n),$ if $(x_0,\dots,x_n)$ and $(y_0,\dots,y_n)$ are two different shortest paths from $x_0=y_0$ to $x_n=y_n.$
\item[(R2)] $(x_0,\dots,x_n)=0,$ if $(x_0,\dots,x_n)$ is a non-shortest path.
\end{itemize}
We say that a non-shortest path $(x_0,\dots,x_n)$ is a minimal non-shortest path, if $(x_0,\dots,x_{n-1})$ and $(x_1,\dots,x_n)$ are shortest paths. It is easy to see that (R2) can be replaced by 
\begin{itemize}
    \item[(R2')] $(x_0,\dots,x_n)=0,$ if $(x_0,\dots,x_n)$ is a minimal 
non-shortest path.
\end{itemize}

The distance $d(x,y)$ between vertices of $G$ as the infimum of lengths of paths from $x$ to $y.$ So $G$ defines a finite quasimetric space with integral distances that we denote by the same letter $G$. We will consider the category of $\ZZ$-graded distance modules $\DMod_\ZZ(G).$  In this section all distance modules are $\ZZ$-graded.

\begin{proposition}\label{prop:magnitude_modules_representation} The ideal $R(G) \triangleleft \KK G$ is admissible and there is an isomorphism of graded algebras
\begin{equation}
\sigma G \cong \KK G/R(G).
\end{equation}
Moreover, there is an equivalence of categories
\begin{equation}
\DMod_{\ZZ}(G)\simeq  \GrRep(G,R(G))
\end{equation}
that sends a $\ZZ$-graded distance module $M$ to a representation $\tilde M$ such that $\tilde M(x)=M(x)$ and $\tilde M((x,y))=M(x,y)$ for each $(x,y)\in G_1.$ 
\end{proposition}
\begin{proof} If there are a least two shortest paths from one vertex to another, then these paths have lengths $\geq 2.$ Non-shortest paths also have length $\geq 2.$ Therefore $R(G) \subseteq J^2.$ There is a finite number of shortest paths in $G.$ If we take $n,$ which is larger than the length of any shortest path, then $J^n\subseteq R(G).$ Hence $R(G)$ is admissible. 

Consider a $\KK$-linear homomorphism $\KK G\to \sigma G$ that sends a shortest path $(x_0,\dots,x_n)$ to $(x_0,x_n),$ and a non-shortest path to zero. In particular, it sends $e_x$ to $(x,x).$ It is easy to check that it is an algebra homomorphism and its kernel is $R(G).$ Therefore $\sigma G\cong \KK G/R(G).$ The equivalence $\DMod_{\ZZ}(G)\simeq  \GrRep(G,R(G))$ follows from Proposition \ref{prop:magmod_equivalence} and equivalence \eqref{eq:graded_rep}.
\end{proof}

As a corollary we obtain the following theorem.

\begin{theorem}
Let $\KK$ be a commutative ring and $G$ be a finite digraph. Then there are isomorphisms
\begin{equation}
\MH_{n,\ell}(G)={\rm Tor}^{\KK G/R(G)}_{n,\ell} (S,S), \hspace{1cm} \MH^{n,\ell}(G)={\rm Ext}_{\KK G/R(G)}^{n,\ell} (S,S). 
\end{equation}
Moreover, the magnitude cohomology algebra is isomorphic to the Yoneda algebra of the left $\KK G/R(G)$-module $S$
\begin{equation}
\MH^{*,*}(G)=\mathcal{Y}^{*,*}_{\KK G/R(G)}(S)
\end{equation}
\end{theorem}
\begin{proof}
It follows from Theorems \ref{th:magnitude_homology_finite}, \ref{th:yoneda} and Proposition \ref{prop:magnitude_modules_representation}.
\end{proof}

\section{Appendix. Graded modules and algebras}

\subsection{Abelian cosmos of graded modules}\label{subsection:cosmos_of_graded_modules}
Let $\KK$ be a commutative ring and $\GGG$ be an abelian group.  In this subsection we will introduce the B\'enabou cosmos of $\GGG$-graded $\KK$-modules $\GGG\text{-}\GrMod(\KK).$ In this subsection the word ``module'' means ``$\KK$-module''.

An $\GGG$-graded module is a family of modules  $A=(A_g)_{g\in \GGG}.$ An element of $A_g$ will be called a homogeneous element of $A$ of degree $g.$ The degree of a homogeneous element $a$ will be denoted by $|a|.$ A morphism $f:A\to B$ of $\GGG$-graded modules of degree $d\in G$ is defined as a family of homomorphisms $(f_g:A_g\to B_{g+d})_{g\in \GGG}.$ The degree of a morphism will be denoted by $|f|=d.$ Therefore for a homogeneous element $a$ of $A$ we have 
$|f(a)|=|f|+|a|.$
The composition of such morphisms is defined in the obvious way so that 
$|gf| = |f|+|g|.$
We denote by $[A,B]_d$ the abelian group of homomorphisms of degree $d$
\begin{equation}
[A,B]_d = \prod_{g_2-g_1=d} \Hom(A_{g_1},B_{g_2}),
\end{equation}
and denote by $[A,B]$ the $\GGG$-graded module, whose components are $[A,B]_d.$ 

For two $\GGG$-graded modules $A$ and $B$ we denote by $A\otimes B$ an $\GGG$-graded module such that 
\begin{equation}
    (A\otimes B)_g = \bigoplus_{g_1+g_2=g}  A_{g_1}\otimes B_{g_2}.
\end{equation}
Then it is easy to check that for any $\GGG$-graded modules $A,B,C$ there is an isomorphism of $\GGG$-graded modules
\begin{equation}
[A\otimes B,C] \cong [A,[B,C]].
\end{equation}

For an $\GGG$-graded module $A$  and a $\ell\in G$ we denote by $A[\ell]$ the shifted $\GGG$-graded module 
\begin{equation}
    (A[\ell])_g = A_{g-\ell}.
\end{equation}
For any (non-graded) module $V$ we also denote by $V[\ell]$ the $\GGG$-graded module concentrated in degree $\ell$ 
\begin{equation}
V[\ell]_g = 
\begin{cases}
V,& g=\ell,\\
0,& g\ne \ell.
\end{cases}
\end{equation} 
It is easy to see that for any $\GGG$-graded module $A$ there are isomorphisms 
\begin{equation}\label{eq:A[ell]}
A[\ell]\cong A\otimes \KK[\ell], \hspace{1cm} A[-\ell] \cong  [\KK[\ell],A].
\end{equation}

The category of $\GGG$-graded modules is a category whose objects are $\GGG$-graded modules and morphisms are morphisms of degree $0.$ 
\begin{equation}
\Hom_{\GGG\text{-}\GrMod(\KK)}(A,B)=[A,B]_0.
\end{equation} 
It is easy to see that the tensor product $\otimes,$ the   inner hom $[-,=]$ and the unit $\KK[0]=\1$ make it an  abelian  B\'enabou cosmos. One can say that it as a category it is a product of the categories of $\KK$-modules $\GGG\text{-}\GrMod(\KK) = \prod_\GGG \Mod(\KK).$ Therefore, this category has enough of projectives and injectives.

\subsection{Graded modules over graded algebras} \label{subsection:graded_modules_over_graded_algebras}
By a $\GGG$-graded $\KK$-algebra $\Lambda$ we mean a collection of $\KK$-modules $(\Lambda_g)_{g\in \GGG}$ equipped with bilinear homomorphisms $\cdot : \Lambda_g\times \Lambda_h\to \Lambda_{g+h}$ and an element $1\in \Lambda_0$ satisfying the associativity and unit axioms (one can say that a graded algebra is a monoid in the cosmos of $\GGG$-graded modules). Traditionally people consider an ordinary algebra $\Lambda^\oplus=\bigoplus_{g\in \GGG} \Lambda_g$ associated with the graded algebra, and call it graded algebra. We will also allow such freedom of speech, replacing $\Lambda$ by $\Lambda^{\oplus}$ if it does not lead to confusion. However, the situation with tensor product and hom of $\GGG$-graded $\Lambda$-modules is more complex, and there such freedom of speech can lead to confusion. So let us describe in details relations between hom-sets and tensor products of $\GGG$-graded $\Lambda$-modules and $\Lambda^\oplus$-modules. Further we assume that the word ``graded'' means ``$\GGG$-graded'' for the fixed $\GGG$. 

If $\Lambda$ is a graded algebra, a right $\GGG$-graded $\Lambda$-module is a collection of $\KK$-modules $M=(M_g)_{g\in \GGG}$ equipped with bilinear maps $\cdot:M_g\times \Lambda_h \longrightarrow M_{g+h}$ satisfying the exterior associativity and unit axioms. A left graded $\Lambda$-module is defined similarly. If $M$ and $N$ are right graded $\Lambda$-modules, then 
\begin{equation}
\Hom_\Lambda(M,N),    
\end{equation}
is a graded $\KK$-module whose $\ell$-th component $\Hom_\Lambda(M,N)_\ell$ consists of collections of $\KK$-homomorphisms $(f_g:M_g\to N_{g+\ell})_{g\in \GGG}$ such that $f_{g+h}(m\cdot a) = f_g(m)\cdot a$ for any $m\in M_g,a\in \Lambda_h.$ The hom-set in the category of graded $\Lambda$-modules is the $0$-component of this graded module
\begin{equation}
\Hom_{\GGG\text{-}\GrMod(\Lambda)}(M,N) = \Hom_\Lambda(M,N)_0. 
\end{equation}
So we can say that 
\begin{equation}\label{eq:hom_d}
\Hom_\Lambda(M,N)_\ell = \Hom_{\GGG\text{-}\GrMod(\Lambda)}(M,N[-\ell]).
\end{equation}

Now assume that $M$ is a right graded $\Lambda$-module and $N$ is a left graded $\Lambda$-module. Then the tensor product 
\begin{equation}
M\otimes_\Lambda N   
\end{equation}
is defined as the graded $\KK$-module, whose $\ell$-th component $(M\otimes_\Lambda N)_\ell$ is defined as the quotient of the $\KK$-module $\bigoplus_{g+h=\ell} M_g\otimes N_h $ by the submodule generated by elements of the form $ma\otimes n - m\otimes an.$ It is easy to see that 
\begin{equation}
\Lambda \otimes_\Lambda N\cong N, \hspace{1cm}  \Hom_\Lambda(\Lambda,M)\cong M.
\end{equation}

For a graded $\Lambda$-module $M$ we denote by $M^\oplus = \bigoplus_{g\in \GGG} M_g$ 
the corresponding 
$\Lambda^\oplus$-module. 
We will also consider a bigger 
$\Lambda^\oplus$-module
$M^\Pi=\prod_{g\in \GGG} M_g.$ 
\begin{equation}
M^{\oplus} \subseteq M^{\Pi}.
\end{equation}
The elements of $M^\Pi$ are families $m_*=(m_g)_{g\in \GGG},$ where $m_g\in M_g.$ 
Then for $a\in \Lambda_h$ the product is defined by $(m_* \cdot a)_{g} = m_{g-h} a.$  
The same notations are used for graded $\KK$-modules. 
Using that the tensor product over $\KK$ commutes with direct sums, it is easy to check that there is an isomorphism  
\begin{equation}\label{eq:tensor_graded-nongraded}
(M \otimes_\Lambda N)^\oplus \cong M^{\oplus} \otimes_{\Lambda^{\oplus}} N^{\oplus}.
\end{equation}
However, for hom-sets the formula is more complicated
\begin{equation}\label{eq:hom_graded-nongraded}
\Hom_{\Lambda^\oplus}(M^\oplus,N^\Pi) \cong  \Hom_\Lambda(M,N)^\Pi
\end{equation}
It follows from the isomorphism  $\Hom_\KK(M^\oplus,N^\Pi)\cong \prod_{g,h}\Hom_\KK(M_g,N_h).$ If both graded modules are non-trivial only in finitely many degrees, then products in the last formula can be replaced by direct sums. 

The derived functors of the functors 
\begin{equation}
-\otimes_\Lambda N, \Hom_\Lambda(M,-) :\GGG\text{-}\GrMod(\Lambda) \longrightarrow \GGG\text{-}\GrMod(\KK)     
\end{equation}
are denoted by 
\begin{equation}
{\rm Tor}_{n,\ell}^\Lambda(M,N) = (L_n(-\otimes_\Lambda N))(M)_\ell, \end{equation}
\begin{equation}
{\rm Ext}^{n,\ell}_\Lambda(M,N)=(R^n\Hom_\Lambda(M,-)(N))_{-\ell}.
\end{equation}
where $(n,\ell)\in \mathbb{N}\times \GGG.$ 
The formula \eqref{eq:hom_d} implies that the bigraded Ext can be presented as the ordinary Ext in the category of graded modules
\begin{equation}\label{eq:ext_bigrad_as_ext_in_cat}
\Ext^{n,\ell}_\Lambda(M,N)=\Ext^n_{\GGG\text{-}\GrMod(\Lambda)}(M,N[\ell]).
\end{equation}
Note that this implies the symmetric description of the bigraded ext
\begin{equation}
\Ext^{n,\ell}_\Lambda(M,N) = (R^n\Hom_\Lambda(-,N)(M))_{-\ell}.
\end{equation}
Using this, formulas \eqref{eq:tensor_graded-nongraded}, \eqref{eq:hom_graded-nongraded} and the fact that for a free graded resolution $P_\bullet$ of $M$ the complex $P_\bullet^{\oplus}$ is a free graded resolution of $M^\oplus$, we obtain 
\begin{equation}
{\rm Tor}_{n,*}^\Lambda(M,N)^\oplus \cong {\rm Tor}_{n}^{\Lambda^\oplus}(M^{\oplus},N^{\oplus}), \hspace{5mm} 
\Ext_{\Lambda^\oplus}^n(M^\oplus,N^\Pi)\cong \Ext^{n,*}_\Lambda(M,N)^\Pi. 
\end{equation}

\printbibliography

\end{document}